\documentclass[11pt,oneside]{amsart}
\usepackage[english]{babel}
\usepackage{fullpage}
\usepackage{amssymb,pstricks,amscd,epsfig}
\usepackage{graphicx}

\usepackage{amsmath, amsthm, amssymb}
\usepackage{pdfsync}
\usepackage[latin1]{inputenc}
\usepackage{stmaryrd}
\usepackage{color}
\usepackage{setspace}
\usepackage[matrix,arrow,tips,curve]{xy}

\numberwithin{equation}{section}

\makeatletter
 \renewcommand\section{\@startsection {section}{1}{\z@}%
     {-4.5ex \@plus -1ex \@minus -.2ex}%
     {2.3ex \@plus.8ex}%
    {\centering\scshape}}

\usepackage{cite}

\newcommand{\Z}{\mathbb{Z}}
\newcommand{\Q}{\mathbb{Q}}
\newcommand{\R}{\mathbb{R}}
\newcommand{\C}{\mathbb{C}}
\renewcommand{\P}{\mathbb{P}}

\newcommand{\Shend}{\mc{E}nd}
\DeclareMathOperator{\End}{End}
\newcommand{\mc}{\mathcal}

\newcommand{\F}{\mc{F}}

\newcommand{\h}{\mc{H}}
\newcommand{\hb}{\overline{\mc{H}}}
\newcommand{\Xb}{\overline{X}}

\newcommand{\Bb}{\overline{B}}
\newcommand{\Zb}{\overline{Z}}

\newcommand{\ov}{\overline}

\newcommand{ \ff} { \frac }

\newcommand{\be}{\begin{equation}}
\newcommand{\ee}{\end{equation}}
\renewcommand{\phi}{\varphi}
\newcommand{\res}[1]{\, \text{\rule[-1ex]{0.1ex}{2.8ex}}_{\, #1}}

\newcommand{\Nt}{\widetilde{N}}

\newcommand{\wt}{\widetilde}

\newtheorem{thm}{Theorem}[section]
\newtheorem*{thmintro}{Theorem}
\newtheorem{prop}[thm]{Proposition}
\newtheorem{lemma}[thm]{Lemma}
\theoremstyle{definition}

\newtheorem{rem}[thm]{Remark}
\newtheorem{defin}[thm]{Definition}
\newtheorem{cor}[thm]{Corollary}

\newtheorem{example}[thm]{Example}

\newtheorem{assumption}[thm]{Assumption}
\newtheorem{conjecture}[thm]{Conjecture}

\numberwithin{equation}{section}

\DeclareMathOperator{\codim}{codim}
\DeclareMathOperator{\id}{id}

\DeclareMathOperator{\Hom}{Hom}
\DeclareMathOperator{\Ext}{Ext}

\DeclareMathOperator{\re}{Re}
\DeclareMathOperator{\tr}{tr}

\DeclareMathOperator{\Sym}{Sym}

\DeclareMathOperator{\Supp}{Supp}
\DeclareMathOperator{\supp}{supp}
\DeclareMathOperator{\Pic}{Pic}
\DeclareMathOperator{\Gr}{Gr}

\DeclareMathOperator{\Jac}{Jac}
\DeclareMathOperator{\NS}{NS}
\DeclareMathOperator{\Fix}{Fix}
\DeclareMathOperator{\im}{Im}
\DeclareMathOperator{\inv}{inv}
\DeclareMathOperator{\Aut}{Aut}
\DeclareMathOperator{\Amp}{Amp}
\DeclareMathOperator{\Res}{Res}
\DeclareMathOperator{\Sing}{Sing}

\author[G. Sacc\`a]{Giulia Sacc\`a}
\address{Mathematics Department\\
Stony Brook University\\
Stony Brook, NY, 11794-3651}
\email{giulia.sacca@stonybrook.edu}


\title[Relative compactified Jacobians and Enriques surfaces]{Relative compactified Jacobians of linear systems on Enriques surfaces}

\begin{document}

\begin{abstract} We study certain moduli spaces of sheaves on Enriques surfaces thereby obtaining, in every odd dimension, new examples of Calabi--Yau manifolds. We describe the geometry (canonical bundle, fundamental group, second Betti number and certain Hodge numbers) of these moduli spaces showing, in partial analogy to the well--known case of sheaves on K3 or Abelian surfaces, how the geometry of the surface reflects that of the moduli space itself.


\end{abstract}

\maketitle

\section*{Introduction}

Moduli spaces of sheaves on K3 surfaces are among the most studied objects in algebraic geometry. Part of their  interest lies in that they inherit the rich structure coming from the K3 surface itself. For example, by work of Mukai \cite{Mukai} the symplectic structure on the surface induces a holomorphic symplectic structure on the smooth locus of the moduli space. When smooth and projective, these moduli spaces provide examples of compact irreducible hyperk\"ahler manifolds \cite{Beauville83}, \cite{Huybrechts2}, \cite{Mukai}, \cite{Kieran7}. On the other hand, not much work has been done regarding the geometry of moduli spaces of sheaves on Enriques surfaces, even though it is natural to expect that their geometry is tightly related to that  of the Enriques surface itself and of the moduli spaces of sheaves on the covering K3 surface.
The present paper describes the geometry of a certain class of moduli spaces of sheaves on an Enriques surface $T$, namely the case of moduli spaces parametrizing pure dimension one sheaves on $T$. By considering pure dimension one sheaves whose support is linearly equivalent to a given curve $C$, these moduli spaces may be viewed as the relative compactified Jacobian of the linear system $|C|$. As such, they have a structure of fibration in abelian varieties.

One of the results of the paper is to show that the canonical bundle of these moduli spaces is \emph{trivial}. Though it is not hard to see that the canonical bundle is a torsion element in the Picard group, it is an interesting surprise that it is actually trivial, and not $2$--torsion as is true in the case of Enriques surfaces. This produces a series of new examples (in every odd dimension) of Calabi--Yau manifolds. Recall that one of the reasons why hyperk\"ahler manifolds have attracted attention is that they are, together with irreducible Calabi-Yau manifolds and complex tori, the building blocks of K\"ahler manifolds with trivial first Chern class \cite{Beauville83}. By proving that the universal cover of these moduli spaces are \emph{irreducible} Calabi--Yau manifolds, this paper thus produces a new series of building blocks for $c_1$--trivial manifolds.

The fact that the properties of these moduli spaces do not fully reflect those of the underlying surface makes the study of their geometry even more compelling.
The other main results of the paper, all of which use the abelian fibration structure, are the computations of the fundamental group, of the cohomology of the structure sheaf, and of the second Betti number.

To fix ideas, let us assume that we are considering a moduli space $N$ parametrizing sheaves whose Fitting support belongs to a linear system $|C|$, of genus  $g \ge 2$. Associating to every sheaf its Fitting support, defines the support (or Le Potier) morphism $N \to |C|=\P^{g-1}$, and endows the moduli space with a fibration structure in the $g$--dimensional Jacobians of the curves belonging to  $|C|$. Notice that these moduli spaces are always non empty. To sum up the main results, we have the following theorem.

\begin{thmintro}
Let $|C|$ be a genus $g \ge 2$ linear system on a general Enriques $T$, let $d \neq g-1$ be an integer, let $H$ be a generic polarization, and let $N \to |C|$ be the component of the moduli space of $H$--semistable sheaves on $T$ with Fitting support in $|C|$ and Euler characteristic equal to $\chi=d-g+1$ that contains sheaves supported on irreducible curves. Suppose the divisibility of $C$ in $NS(T)$ is coprime with $2(d-g+1)$ then,
\begin{enumerate}
\item $N$ is a smooth $(2g-1)$-dimensional  Calabi--Yau variety, i.e.,
\[
\omega_{N}\cong {\mc O}_{N}, \quad \text{and  } \quad h^{p,0}(N)=0 \quad \text{ for } \,\,\, p\neq 0, 2g-1.
\]
\item There is a surjection $ \Z/(2) \twoheadrightarrow \pi_1(N)$ which, under a natural assumption that holds in many cases (e.g. for low values of $g$ and in the case $|C|$ is a primitive linear system) and that is expected to hold in general (see Assumption \ref{assumption} and subsequent discussion),  turns out to be an isomorphism.
\item Under the same assumption, we show that for $g \ge 3$ 
\[
h^2(N)=11.
\]
\item For $g=2$, we get Calabi--Yau $3$--folds with the following Hodge diamond
\[
\begin{array}{cccccc}
&1&&\\
&0\,\,\,\,\,\,0&&\\
\,\,\,\,0&10&0\,\,\,\,&\\
1\,\,\,\,&10\,\,\,\,\,\,10&\,\,\,\,1&\\\
\end{array}
\]
\end{enumerate}
\end{thmintro}

Geometrically, we can realize the universal cover $\wt N$ of $N$ via the Stein factorization of the norm map. Since $\wt N$ is simply connected, it is an \emph{irreducible} Calabi--Yau manifold of dimension $2g-1$.

Further content of the paper regards the support morphism which, though appearing ubiquitously in algebraic geometry (e.g. the Hitchin system, the Beauville--Mukai integrable system) is not very well understood (especially over the locus of non--reduced curves). We give some factual and conjectural properties for the support morphism in the case of a linear system on a smooth projective surface  (Section \ref{fibration in jacobians}). We also mention a recent result of Yoshioka \cite{yoshioka-enriques} regarding this morphism in the case of primitive linear systems.
In the specific case of Enriques surfaces, we study the corresponding variation of Hodge structures (via degeneration of Hodge bundles) and compute the push--forward of the structure sheaf (\ref{higher direct images}). This involves monodromy calculations that are of independent interest \ref{irreducible monodromy non hyp}.

Beyond the case of the Hilbert scheme $T^{[n]}$ of $n$ points on $T$, whose canonical bundle (which is shown to be of $2$--torsion) and fundamental group are computed in \cite{Oguiso_Schroer11}, this paper is the first one that studies geometric properties of a moduli space of sheaves on an Enriques surfaces. The previously existing literature,  \cite{Kim98}, \cite{Kim06}, and \cite{Zowislok}  studies smoothness and irreducibility properties of moduli spaces of sheaves on Enriques surfaces, by realizing them as double covers of Lagrangian subvarieties of moduli spaces on the covering K3 surface (and after the first version of this paper appeared, also \cite{Nuer} where non emptiness is  studied and \cite{Nuer2} where the birational geometry of these moduli spaces is studied). Finally,  \cite{Hauzer} finds an explicit parametrization in the case when the moduli spaces are one--dimensional and shows how to relate moduli spaces of sheaves of arbitrary rank to those of low rank.

The techniques of this paper can be used also to study relative compactified Jacobians of linear systems on bi-elliptic surfaces. These moduli spaces produce another series of Calabi-Yau manifolds, whose geometry is the subject of a forthcoming paper by the author.

\subsection*{Acknowledgements} It is my pleasure to thank E. Arbarello, A. Bruno, M. de Cataldo, D. Huybrechts, J. Koll\'ar, L. Migliorini, M. Popa, V. Shende, C. Voisin and my advisor G. Tian for useful and fun conversations, and I. Dolgachev for a written communication. I would also like to express my gratitude  to K. Yoshioka for directing my attention to his paper \cite{yoshioka-enriques} and to the anonymous referee who pointed out several inaccuracies and a mistake in the first version of this paper and whose suggestions have greatly improved the paper. Part of this work was done while visiting the Beijing International Center for Mathematical Research in March 2010 and versions of the manuscript were prepared while visiting Bonn University in July 2012 and the Institut de Math\'ematique de Jussieu in October 2012. I am grateful to these institutions for the warm hospitality I received and for the great working conditions. This work is part of the author's Ph.D. thesis which received the Enriques prize, awarded by the Unione Matematica Italiana. I heartily thank this institution for honoring me with such a prize.

\section{Preliminaries} \label{prelims}

\subsection{Set up and notation}  \label{notation}
\subsubsection{}Throughout the paper $T$ will denote an Enriques surface, that is a smooth projective surface with
\[
H^1(T, \mc O_T)=0,
\]
and whose canonical bundle $\omega_T$ defines a non-trivial $2$-torsion element of the Picard group.  It is well known that $\pi_1(T)=\Z/(2)$ and that its universal cover, which will be denoted by $S$, is a K3 surface. The covering morphism
\[
f: S \to T,
\]
is the double covering induced by $\omega_T$. The deck involution 
\[
\quad \iota: S \to S,
\]
is antisymplectic, i.e., if $\sigma$ denotes the holomorphic symplectic form on $S$, then $\iota^* \sigma=-\sigma$.

By $C$ we will denote a curve in $T$.
Using the Riemann-Roch theorem one can see that if $C^2 \ge0$, then the line bundle ${\mc O}(C)\otimes \omega_T$ is also effective. We will denote by $C'$ a curve in $|\mc O(C)\otimes \omega_T|$.
We also set
\[
D:=f^{-1}(C) \subset S,
\]

By the Hodge index theorem, if the arithmetic genus $g$ of $C$ satisfies $g \ge 2$ and $C$ is connected and reduced, then the covering
\[
f: D \to C,
\]
is connected. In particular, the two-torsion line bundle 
\be \label{eta}
\eta:={\omega_T}_{|C},
\ee
is not trivial.
When this is the case the genus of $D$ is equal to
\[
h= 2g-1.
\]
If $g\ge 2$, then $\dim |C|=\dim |C'|=g-1$ and $\dim |D|=h$.
Moreover, we see that $|C|$ and $|C'|$, identified with their images in $|D|$ under the pullback morphism, 
are the two $\iota$-invariant linear subspaces of $|D|$.

\subsection{Pure dimension one sheaves on surfaces}
\subsubsection{}Let $(X,H)$ be a smooth projective polarized surface. Associated to any coherent sheaf $F$ on $X$, is the ideal sheaf
\[
\mc I_F:=\ker [\mc O_X \to \Shend(F)],
\]
defining the $\mc O_X$-module structure on $F$. The \emph{scheme theoretic support} of $F$, denoted $\Supp(F)$, is the scheme defined by $\mc I_F$. 
A sheaf is called \emph{pure of dimension $d$} if for any subsheaf  $0 \neq G \subseteq F$, $\dim \Supp(G)=d$ . Let $F$ be a pure dimension one sheaf on $X$, then $\Supp(F) $ is a (possibly non-integral) curve, and $F=i_*L$, where $i: \Supp(F) \to X$ is the natural embedding and where $L$ is a sheaf on $\Supp(F)$ having no subsheaves  that are supported on points.\\
For pure dimension one sheaves, we will also consider another type of support, the \emph{Fitting support}, which is defined in the following way. A pure dimension one sheaf on a smooth projective has homological dimension one (\cite{Huybrechts-Lehn}, Chapter 1), i.e., there exists a length one locally free resolution of $F$,
\[
0 \to L_1 \stackrel{a}{\rightarrow} L_0 \to F \to 0.
\]
The Fitting support of $F$, denoted $\supp(F)$, is the subscheme of $X$ defined by the equation $\det a=0$. Contrary to case of the scheme theoretic support, the Fitting support behaves well in families. It is important to point out that the class in cohomology of the pure dimension one scheme $\supp(F)$ is exactly the first Chern class $c_1(F)$.

\subsubsection{} For pure dimension one sheaves stability with respect to the Hilbert polynomial defined by $H$ amounts to considering stability with respect to the slope function
\[
\mu_H(F)=\ff{\chi(F)}{c_1(F) \cdot H},
\]
where $\chi(F)$ denotes the Euler characteristic of $F$. So if $F$ is pure of rank one  supported on a reduced curve $\Gamma$, $F$ is $H$--semistable if and only if for every subcurve $\Gamma' \subset \Gamma$ we have
\be \label{stability pure rank one sheaves}
\ff{\chi(F)}{\Gamma \cdot H} \le  \ff{\chi(F_{\Gamma'})}{\Gamma' \cdot H},
\ee
where
\[
F_{\Gamma'}:=F \res{\Gamma'} \slash Tor(F \res{\Gamma'} )
\]
is the restriction of $F$ to the subcurve, modulo its torsion. We say that $H$ is $\chi$--general for a curve $\Gamma$ and an integer $\chi$ (or $d$--general for $d:=\chi-\chi(\mc O_\Gamma)$) if for every subcurve $\Gamma'\subset \Gamma$ the rational number $ \chi  \ff{\Gamma' \cdot H}{\Gamma \cdot H}$ is not an integer. This guarantees that $H$--semistability is equivalent to $H$--stability.
Recall that for any coherent sheaf $F$ on $X$, one can define the Mukai vector $v(F) \in H^*_{alg}(T, \Z)$. When $F$ is pure of dimension one it is given by
\[
v=v(F)=(0,c_1(F), ch_2(F))=(0, c_1(F), \chi(F)-\ff{1}{2}c_1(X)c_1(F)).
\]
Here $ch_2(F)$ denotes the degree two part of the Chern character of $F$. Let
\[
M_{v,H}(X)
\]
be the moduli space of $H$-semistable sheaves on $X$ with Mukai vector $v$.
Let $F$ be a sheaf with $v(F)=v$ and let $\mc H_{c_1(v)}$ be the Hilbert scheme parametrizing subschemes with  cohomology class equal to $\supp(F)$.
Since the Fitting support behaves well in families, we can define the Le Potier \cite{LePotier} or support morphism
\[
\begin{aligned}
\pi: M_{v,H}(X) & \longrightarrow \mc H_{{c_1(v)}},\\
F & \longmapsto \supp(F)
\end{aligned}
\]
which associates to a pure sheaf of dimension one its Fitting support. 
For a curve $\Gamma \subset X$ defining a point $[\Gamma] \in \mc H_{c_1(v)}$, the fiber $M_{v,H}(X)_{[\Gamma]}:=\pi^{-1}([\Gamma])$ is the Simpson moduli space of $H_{|\Gamma}$-semistable sheaves on $\Gamma$. For example, the fiber over a nodal curve is isomorphic to an appropriate compactfied Jacobian in the sense of Oda and Seshadri \cite{Oda_Seshadri79} (cf. Alexeev \cite{Alexeev04}).

If $h^1(X, \mc O_X)=0$, then every component of $\mc H_{c_1(v)}$ is just the linear system of some line bundle with that cohomology class. Hence, if  $v=(0,\Gamma, \chi)$  and we let 
\[ 
M_{v,H}(X, |\Gamma|)
\]
be the irreducible component of $\pi^{-1}(|\Gamma|)$ containing the locus of locally free sheaves with integral Fitting support, we can think of
\be \label{support morphism}
\pi: M_{v,H}(X, |\Gamma|)=\ov J_{H,d}(|\Gamma|) \to |\Gamma|
\ee
as the relative compactified Jacobian of degree $d=\chi-\chi(\mc O_\Gamma) $ of the linear system $|\Gamma|$. Indeed, the fiber over a reduced curve $\Gamma$ is just the degree $d$ compactified Jacobian of that curve with respect to the polarization $H$, i.e.
\[
\pi^{-1}([\Gamma])=\ov J_{H,d}(\Gamma)
\]
where $d=\chi-\chi(\mc O_\Gamma)$. Notice that if the curve is integral then the compactified Jacobian does not depend on $H$.

\begin{lemma} \label{existence stable sheaves}
If $\Gamma$ is an integral curve, the fiber $\pi^{-1}([\Gamma])$ is contained in the stable locus $M^s_{v,H}(X, |\Gamma|)$ of $M^s_{v,H}(X, |\Gamma|)$. In particular, if there exist an integral curve in the linear system $|\Gamma|$, then $M^s_{v,H}(X, |\Gamma|)$ is non-empty and changing polarization only changes $M_{v,H}(X, |\Gamma|)$ within its birational class.
\end{lemma}
\begin{proof}
If $\supp(F)$ is integral then there is no condition (\ref{stability pure rank one sheaves}) to be checked, i.e. any surjection $F \to G$ is an isomorphism.
\end{proof}

Often in this paper if we consider the restriction of (\ref{support morphism}) to a locus in $|\Gamma|$ parametrizing irreducible curves, we omit the polarization from the notation.

If the curve $\Gamma$ is reducible, then stability depends on the degree of $H$ on each component of $\Gamma$. For later use, we work out the characterization of semistability for pure rank one sheaves supported on  a curve that is the union of two smooth components meeting transversally.

\begin{lemma} \label{stability on two components}
Let $\Gamma=\Gamma_1+\Gamma_2$ be a curve that is the union of two smooth components meeting transversely in $\delta$ points. Let $F$ be a pure rank one sheaf on $\Gamma$  with $\chi(F)=\chi$ and let $\delta' \le \delta$ be the number of nodes where $F$ is locally free. Then $F$ is $H$--semistable if and only if
\[
\ff{h_1}{h} \chi \le \chi_1 \le \ff{h_1}{h} \chi+\delta',
\]
where $h_1=H \cdot \Gamma_1$, $h_2=H \cdot \Gamma_2$, $h=h_1+h_2$, and where $\chi_i=\chi(F_{\Gamma_i})$. Furthermore, $F$ is $H$-stable if and only if the inequalities are strict. As a consequence, if $H$ is general then semistability is equivalent to stability.
\end{lemma}
\begin{proof}
This follows readily from (\ref{stability pure rank one sheaves}) and the fact that $F$ fits into a short exact sequence
\[
0 \to F \to F_{\Gamma_1} \oplus F_{\Gamma_2} \to \C^{\delta'} \to 0
\]
so that $\chi+\delta'=\chi_1+\chi_2$ and we can rewrite inequality (\ref{stability pure rank one sheaves}) for $i=2$ in terms of $h_1$ an $\chi_1$.
\end{proof}

We will also need the following important result by Melo, Rapagnetta, and Viviani.

\begin{prop} \label{MRV on compactified Jac} \cite{Melo-Rapagnetta-Viviani14} Let  $\Gamma$ be a reduced locally planar curve of genus $g$, let $d$ be an integer, let $H$ be a $d$--general polarization for $\Gamma$, and let $\ov J_{H,d}(\Gamma)$ be the compactified Jacobian of degree $d$. Then $\ov J_{H,d}(\Gamma)$ is l.c.i of dimension $g$ and its smooth locus is precisely the locus
\[
J_{H,d}(\Gamma)
\]
parametrizing line bundles.
\end{prop}

As a consequence of these considerations we highlight the following well known Corollary that will be used in Section \ref{fundamental group}.

\begin{cor} \label{number of components}
Let $\Gamma=\Gamma_1+\Gamma_2$ be the union of two smooth components meeting in $\delta$ points. Then $\ov J_{H,d}(\Gamma)$ has $\delta$ irreducible components, parametrized by the $\delta$ pairs $(\chi_1, \chi_2)$ satisfying the condition of Lemma \ref{stability on two components}: For every such pair $(\chi_1, \chi_2)$ the corresponding component contains as a dense open subset the locus of line bundles whose restriction to $\Gamma_1$ and $\Gamma_2$ have Euler characteristic $\chi_1$ and $\chi_2$, respectively. 
\end{cor}

Suppose now that $\chi(F)=-g+1$, where $g$ is equal to the genus of $\Gamma=\supp(F)$. In this case, there is a rational section
\be \label{section on irreducible}
s: \mc H_{c_1(v)} \dasharrow M_{v,H}(X),
\ee
defined on an open subset containing integral curves.

\begin{rem} \label{section defined} If $\Gamma$ is ample then we can also consider $H=\Gamma$ and it is not hard to see using (\ref{stability pure rank one sheaves}) that the structure sheaf of every curve is stable. It follows that the section is a regular morphism. This guarantees that there is a non--empty open set in the ample cone of polarizations for which the section is a regular morphism.
\end{rem}

\subsubsection{} Recall that if $F$ is a stable sheaf, the tangent space to the moduli space at a point $[F]$ is canonically isomorphic to $\Ext^1(F,F)$. 
Moreover (\cite{Artamkin88}, \cite{Mukai}), the obstructions to deforming $F$ on $X$ lie in
\be \label{obstructions}
\Ext^2(F, F)_0:=\ker[ \tr: \Ext^2(F,F) \to H^2(X, \mc O_X)],
\ee
where $\tr: \Ext^2(F,F) \to H^2(X, \mc O_X)$ is the trace morphism (cf. \cite{Huybrechts-Lehn}). Hence, by Serre duality, if $(X, H)$ is a polarized surface with $\omega_X=\mc O_X$ (i.e. a K3 or an abelian surface) and $F$ is a pure sheaf $H$-stable sheaf on $X$ with $v=v(F)$ then the moduli space $M_{v,H}$ is smooth at the point $[F]$.

\begin{thm}[\cite{Mukai},\cite{Gottsche-Huybrechts}, \cite{Kieran7},\cite{Huybrechts2},\cite{Yoshioka99}, \cite{Yoshioka}] \label{theorem-definition walls}
Let $X$ be a K3 surface, let $H$ be a polarization on $X$, let $F$ be a pure sheaf on $X$ and set $v=v(F)$. The locus $M_{v,H}^s(X)$ of $H$-stable sheaves on $X$ has a holomorphic symplectic form. If $v$ is primitive, then there exists a locally finite collection of real codimension one linear subspaces (called the $v$-walls) in the ample cone $\Amp(X)\otimes \Q$ such that, if $H$ is chosen outside the union of the $v$-walls then $H$--stability coincides with $H$--semistability (an $H$ satisfying this condition is called $v$-generic, see Definition \ref{definition of walls} below) and $M_{v,H}$ is an irreducible holomorphic symplectic manifold of K3$^{[n]}$-type.\footnote{By definition, this means that it is deformation equivalent to a Hilbert scheme of points on a K3 surface.} 
\end{thm}

Hence, if $v=(0, \Gamma, \chi)$ and $H$ is chosen to be $v$--generic, then $H$ is $\chi$--general for every curve in $|\Gamma|$.

Finally, we will need the following expression for the symplectic form on $M_{v,H}$: 
on the tangent space $T_{[F]} M_{v,H}(X)=\Ext^1(F,F)$ at a point $[F]$ corresponding to a stable sheaf $F$, the symplectic form is given by the composition
\be \label{symplectic form}
\begin{aligned}
\sigma:  \Ext^1(F, F) \times \Ext^1(F, F) & \stackrel{\cup}{\longrightarrow} & \Ext^2(F, F) & \stackrel{\tr}{\longrightarrow} H^2(S, {\mc O}_S)\cong \C, \\
(e, f) & \longmapsto & e\cup f \,\,\,\,\,\,\, & \longmapsto  \,\,\,\,\, \tr(e\cup f)
\end{aligned}
\ee
where the identification $H^2(S, {\mc O}_S)=H^2(S, \omega_S) \cong \C$ is Serre dual to the isomorphism
\be \label{isomorphism induced by sigma}
H^0(S, \omega_S)=\C \sigma \cong \C,
\ee
defined by the choice of a, unique up to scalar, symplectic form $\sigma \in H^0(S, \omega_S)$.

\subsection{Some facts about linear systems on Enriques surfaces.} \label{linear systems enriques}

In this section we collect a few (mostly known) results about linear systems on an Enriques surface that will be needed in the rest of the paper. After a general introduction, we will focus on linear systems on a general Enriques surface. For a more complete treatment we refer to \cite{Cossec} and \cite{Cossec-Dolgachev}.

\subsubsection{} \label{primitive} \label{general enriques 1} The N\'eron-Severi group\footnote{Following Definition 1.1.13 of \cite{Lazarsfeld1}, we let the N\'eron-Severi group of a smooth projective variety $X$ be the group of line bundles on $X$ modulo numerical equivalence. In particular, it is torsion free.}  of $T$ has rank $10$ and is isomorphic to the abstract lattice $U\oplus E_8(-1)$, where $U$ and $E_8(-1)$ denote the hyperbolic lattice and the positive definite $E_8$ root lattice, respectively.
The pullback homomorphism
\be
f^*: \NS(T) \to \NS(S),
\ee
is injective and in \cite{Namikawa1} it is shown that the image of the N\'eron-Severi group of $T$ in the N\'eron-Severi group of $S$ is a primitive sub lattice (of rank $10$). In particular, if we choose $C$ so that its class is primitive in $\NS(T)$, then so is the class of $D$ in $\NS(S)$. By abuse of notation, we say that a curve or a line bundle is primitive if its class in the N\'eron-Severi group is a primitive element of the lattice.
Moreover (\cite{Namikawa1}, by choosing the Enriques surface general in moduli we can ensure that
\be \label{NS general enriques}
f^*(\NS(T))=\NS(S).
\ee
When this is the case, $\iota^*$ acts as the identity on $\NS(S)$ and there are no smooth rational curves on $S$ or on $T$.
In particular, there are not effective line bundles of negative self intersection.

From now on, when we say that $T$ is general we will assume that (\ref{NS general enriques}) holds.

\begin{lemma} Let $T$ be a general Enriques surface.
\begin{enumerate}
\item If $L$ is a line bundle on $T$ with $L^2>0$ (respectively $L^2\ge 0$), then $\pm L$ is ample (respectively nef). 
\item If $L^2>0$ then the general member of $|L|$ is irreducible.
\end{enumerate}
\end{lemma}
\begin{proof}
Item (1) follows immediately from the Hodge index theorem and the fact that there are no curves with negative self intersection on a general $T$. Item (2) is \cite[Prop. 3.1.6]{Cossec-Dolgachev} and \cite[Cor 3.1.2]{Cossec-Dolgachev}.
\end{proof}




\begin{lemma} \label{irreducibility} Let $T$ be a general Enriques surface, and let $C \subset T$  be a primitive curve of genus $g \ge 2$. If $C$ is  irreducible, then so is its preimage $D=f^{-1}(C)$.
\end{lemma}
\begin{proof}
Suppose $D$ breaks into the sum of
two irreducible components $D_1$ and $D_2$. Since $D_1$ and $D_2$ are interchanged by the involution and since by assumption $\iota^*$ acts as the identity on $\NS(S)$, it follows that $D_1 \sim D_2$ and hence that $D\sim 2D_i$, contradicting the fact that $C$, and thus $D$ by \ref{primitive}, is primitive.
\end{proof}

We should also point out that for the conclusion of this lemma to hold, it is sufficient to assume that the class of $C$ is not divisible by $2$.

Contrary to what happens for positive genus where the dimension of an effective linear system of genus $g$ is equal to $g-1$, if $L$ is an effective linear system with $L^2=0$, then the dimension of $|L|$ depends on the divisibility of the class of $L$ in $\NS(T)$. It is well known \cite[\S 1.6]{Cossec} that if $L$ is a primitive, then $L=\mc O(e)$ for a primitive elliptic curve and $|L|=\{e\}$. According to the notation introduced in Subsection \ref{notation}, we denote by
\[
e'
\]
the (unique) curve in the linear system $L \otimes \omega_T$. If $L=\mc O(2e)$, then $|L|$ is a pencil, whose general fiber is a smooth elliptic curve and that has exactly two double fibers
\[
e \quad \text{ and } \quad e'.
\]
This shows that  the canonical bundles of $T$ is equal to the difference of the two half fibers, i.e., $\omega_T=\mc O(e-e')$.

From these considerations one can deduce (see \cite[Thm 1.5.1]{Cossec}) that if $L=\mc O(ke)$ with $k \ge 2$, then
\[
|L|\cong \Sym^{\lfloor \ff{k}{2} \rfloor}|2e|
\]
and $L$ has a fixed component if and only if $k$ is odd. Primitive elliptic curves and elliptic pencils play a big role in the study of linear systems on Enriques surfaces.

\begin{defin}
A genus $g \ge 2$ linear system $|C|$ on an Enriques $T$ is called hyperelliptic if $g=2$ or if the map $T \dashrightarrow \P^{g-1}$ defined by $|C|$ is of degree $2$ onto a rational normal scroll of degree $g-2$ in $\P^{g-1}$.
\end{defin}

The following proposition gives a very useful characterization of hyperelliptic linear systems

\begin{prop}\cite[Prop. 4.5.1 and Cor. 4.5.1]{Cossec-Dolgachev} \label{hyperelliptic enriques}
Let $T$ be an Enriques surface, and let $C\subset T$ be an irreducible curve with $C^2=2g-2 \ge 2$. The following are equivalent,
\begin{enumerate}
\item The linear system $|C|$ is hyperelliptic;
\item $|C|$ has base points;
\item There exists a primitive elliptic curve $e_1$ such that $C\cdot e_1 =1$.
\end{enumerate}
Moreover, if $T$ is general then $|C|$ is hyperelliptic if and only if $C\equiv (g-1)e_1+e_2$ for two primitive elliptic curves with $e_1 \cdot e_2=1$.
\end{prop}

Up to tensoring by $\omega_T$, the two elliptic curves are determined by $|C|$.

It is worth mentioning here that if $|C|$ is hyperelliptic then its base locus (which is non empty by the proposition above) consists of two simple points, described by the following Lemma 
\begin{lemma} \label{base points}
Let $C=(g-1)e_1+e_2$ be a hyperelliptic linear system. The two simple base points of $|C|$ are 
\[ \begin{aligned}
e_1 \cap e_2' \text{ and } e'_1 \cap e_2 & \text{ if } g-1 \text{ is odd }\\
e_1 \cap e_2 \text{ and } e'_1 \cap e_2 & \text{ if } g-1 \text{ is even }
\end{aligned}
\]
\end{lemma}
\begin{proof}
The proof is straightforward after noticing that $\mc O(C) \res{e_1}$ is equal to $\mc O(e_2') \res{e_1}$ or $\mc O(e_2) \res{e_1}$ depending on whether $g-1$ is odd or even, and similarly for the restriction to $e_1'$.
\end{proof}


\begin{cor} \label{section}
Let $p: \mc C \to |C|$ be the universal family of curves of a hyperelliptic linear system. Then $\mc C$ is smooth and $p$ has two sections.
\end{cor}
\begin{proof}
By the Lemma above, $|C|$ has two simple baspoints and hence we can identify this linear system as a base point free linear system on the blow up $\wt T$ of $T$ at the two base points. Since the universal family $ \mc C \to T$ factors via the blow up morphism $\wt T \to T$, we see that $\mc C$ is smooth. The statement about the sections also follows readily from the Lemma above.
\end{proof}

 Finally, it is known that the general curve in $|C|$ is a smooth hyperelliptic curve \cite[Cor 4.5.1]{Cossec-Dolgachev}.

\begin{cor} \label{smooth general member}
Let $L$ be an effective line bundle on a general $T$ with $L^2>0$. Then the general member of $|L|$ is a smooth connected curve.
\end{cor}

We will need the following observation regarding intersection of curves on a general $T$. If $C$ and $\Gamma$ are two curves such that $C \cdot \Gamma=1$, then one of them has to be a primitive elliptic curve and the other one either another primitive elliptic curve or a hyperelliptic curve: If they are both of genus $1$ there is nothing to show. So let us suppose that one of them, say $C$, is of genus $g  \ge 2$. Up to moving $C$ in its linear system we can assume it to be smooth. Suppose that $\Gamma$ is not primitive elliptic so that it moves in a positive dimensional linear system. If $|\Gamma|$ is base point free, then $\mc O(\Gamma) \res{C}$ would cut a positive dimensional degree one linear system on $C$,  providing a contradiction since $C$ is not rational. If $|\Gamma|$ has a base point, then by the proposition above it has to be hyperelliptic. It follows that there are two primitive elliptic curves such that $|\Gamma|=|ne_1+e_2|$, with $n=g(\Gamma)-1$. But since by the Hodge index theoerm $C \cdot e_i >0$, for $i=1$ and $2$, we get a contradiction to the fact that $C \cdot \Gamma=1$.

%

%


In the rest of the paper we will need some knowledge about singular curves in linear systems on a general $T$. Given a linear system $|C|$ on $T$ with smooth connected general member, we define the discriminant
\[
\Delta \subset |C|
\]
of $|C|$ to be the closed codimension one subset of $|C|$ parametrizing singular members (for the sake of this paper it will be enough to consider $\Delta$ with its reduced induced structure). The following three propositions describe the curves parametrized by the general points of the discriminant.

\begin{prop} \label{reducible in codimension one}
Let $|C|$ be a genus $g \ge 2$ linear system on a general Enriques surface $T$. Then $|C|$ has reducible members in codimension one  if and only if $|C|$ is hyperelliptic. \end{prop}
\begin{proof}
Consider a reducible member of the form
\be \label{decomposition}
C_1+C_2,
\ee
where $C_1$ and $C_2$ have no common components
and set $\nu=C_1\cdot C_2$. 
For $i=1,2$, we let $g_i$ be the arithmetic genus of $C_i$. We have
\[
g=g_1+g_2+\nu-1,
\]
so $\dim |C|=g_1-1+g_2-1+\nu$.
Since the irregularity of $T$ is zero, the locus of curves in $|C|$ having a decomposition like that in (\ref{decomposition}) admits a finite surjective morphism from the product of linear systems $|C_1|\times |C_2|$.
Case by case, we will compare the dimension of $|C_1|$ and $|C_2|$ with that of $|C|$. Clearly, if $\dim |C_i|=g_i-1$  for both $i=1$ and $2$ (as is the case the two curves have either positive genus greater than one or are primitive elliptic), then
\[
\dim |C|= \dim |C_1|+ \dim |C_2|+\nu=\dim |C_1|\times |C_2|+\nu.
\] 

It follows that the codimension of the locus of curves of this type is equal to one if and only if $\nu=1$. If this is the case, then by the remarks following Corollary \ref{smooth general member} at least one of the two curves, say $C_1$, has to be a primitive elliptic curve. This implies that $C \cdot C_1=1$ and hence  by the remarks following Corollary \ref{smooth general member}, that $|C|$ is hyperelliptic.

Next, consider the case where $C_1\in |s e_1|=\P^{\lfloor \ff{s}{2}\rfloor}$ for some primitive elliptic curve $e_1$ and some integer $s \ge 1$. If $g_2 \ge 2$, we have $\dim |C_1|\times |C_2|=\lfloor \ff{s}{2}\rfloor+g_2-1$ and $\dim |C|= g_2-1+ \nu$. If $C_2=te_2$ with $t \ge 1$, then $\dim |C_1|\times |C_2|=\lfloor \ff{s}{2}\rfloor+\lfloor \ff{t}{2}\rfloor$ and $\dim|C|=\nu$.
Either way we have
\[
\dim  |C| - \dim (|C_1|\times |C_2|)=\left\{\begin{aligned} & \nu-\lfloor \ff{s}{2}\rfloor, && \text{if} \,\, g_2\ge 2\\ &\nu-\lfloor \ff{s}{2}\rfloor-\lfloor \ff{t}{2}\rfloor, && \,\text{if} \,\, g_2=1. \end{aligned}   \right.
\]
In the first case, since $\nu=s\nu'$ and $\nu' \ge 1$, then we are done, unless $a)$ $s=1$ and $\nu'=1$, or $b)$ $s=2$ and $\nu'=1$. In case $a)$, $C_1$ is primitive elliptic and $C \cdot C_1=1$, so we are in the hyperelliptic case. In case $b)$, $C_2 \cdot e_1=1$ so the   curve $C_2$ is  hyperelliptic and we can write it as $\nu e_1+e_2$, with $e_1\cdot e_2=1$. It follows that $|C|=|(\nu+s)e_1+e_2|$ is hyperelliptic. 

As for the second case, we can set $\nu=st\nu'$. It follows that we are done unless $\nu'=1$, $s=2$ and $t=1$ (or $s=1$ and $t=2$). This means that $C_2=e_2$, with $e_1\cdot e_2=1$ and that $|C|=2e_1+e_2$ is hyperelliptic (or $|C|=|2e_2+e_1|$).

\end{proof}

\begin{prop}  \label{Verra}
Let $|C|$ be a genus $g \ge 3 $ non--hyperelliptic linear system on a general Enriques surface. If $|C|\neq |2(e_1+e_2)|$ for two primitive elliptic curves $e_1$ and $e_2$ with $e_1 \cdot e_2=1$, then there is an open dense subset of the discriminant parametrizing irreducible curves with one single node. 
\end{prop}

\begin{proof}
It is well known that if $|C|$ is very ample then $|C|$ contains a Lefschetz pencil (e.g. \cite[\S 2.1 II]{Voisin2}), so we only need to prove the statement in the case where $|C|$ is not very ample. By \cite[Thm 1.2]{Knutsen} a curve $C$ of genus $g \ge 2$ on a general Enriques surface is very ample if and only if  there is no primitive elliptic curve $e$ such that $C \cdot e  =1$ or $C \cdot e  =2$. Since in the first case $|C|$ is hyperelliptic, we only need to prove the statement in the second case.
 Under this assumption, we claim that the linear system $|C|$ satisfies one of the following
\begin{enumerate}
\item $|C|$ defines a degree $4$ morphism $\psi: T \to \P^2$.
\item $|C|$ defines a degree one morphism $T \to R \subset \P^{g-1}$ onto a non normal surface.
\item $|C|=|2(e_1+e_2)|$ for some primitive elliptic curves $e_1$ and $e_2$ with $e_i \cdot e_j=1$, for $i\neq j$.
\end{enumerate}
Indeed, by \cite[Theorem 4.6.3]{Cossec-Dolgachev} either (1) happens, or $|C|$ defines a birational morphism onto a non normal surface with double lines, or $|C|$ is base point free and defines a degree two map. These linear systems are called superelliptic (see page 228 of \cite{Cossec-Dolgachev}). If this is the case, Proposition 4.7.1, Theorem 4.7.1, and Theorem 4.7.2 of \cite{Cossec-Dolgachev} imply that case (3) occurs (case (i) of Thm. 4.7.2 is excluded using the fact that $T$ contains no rational curves). For the first two cases above, let us now prove that in codimension one only curves with one single node can occur. Case (1) uses the description of the ramification locus of $\psi$ provided in \cite{Verra-83}. Verra shows that, generically, the ramification locus is equal to a degree $12$ curve $\Gamma \subset \P^2$ that has $36$  cuspidal points and no other singularity. Since lines that are tangent to a smooth point of $\Gamma$ appear only in codimension one and since a plane curve has only a finite number of bitangents or flexes, we only need to show that if $\ell$ is a general line through a cusp $\gamma \in \Gamma$, then $\psi^{-1}(\ell)$ has at worst one node. This is proved in Lemma \ref{cuspidi} below.

Let us pass to case (2). Since the map is birational, a curve in $|C|$ can be singular only if it is the preimage $\psi^{-1}(H)$ of a hyperplane section $H$ of $R$ that  either is tangent to $R$ at a smooth point (or is the limit of such), or passes through the singular locus of $R$. We can argue as in \cite[\S 2.1 II]{Voisin2} (see remarks after Cor. 2.8 of loc. cit, which apply to a smooth quasi--projective  variety) and conclude that the general hyperplane that is tangent to the smooth locus of $R$ has one single ordinary double point (i.e. hyperplanes that are have two ordinary double points or other singularities appear in codimension $2$). We are left with analyzing what happens over the hyperplane sections through the singular locus $\Sing R$. Notice that, by Bertini, these hyperplane sections are smooth outside of $\Sing R$ so  we only have to understand what happens over the singular locus.

To do so, we first have to understand what the singular locus of $R$ looks like: using Propositions 3.6.2 and 3.6.3 of \cite{Cossec-Dolgachev} we may assume that $|C|$ is one of the following
\begin{enumerate}
\item[(a)] $|ke+2e_1|$, where $e$ and $e_1$ are primitive elliptic curves with $e \cdot e_1=1$ and $k \ge 3$ (genus $g=2k+1$).
\item[(b)] $|ke+e_1|$, where $e$ and $e_1$ are primitive elliptic curves with $e \cdot e_1=2$ and $k \ge 2$ (genus $g=2k+1$).
\item[(c)] $|ke+e_1+e_2|$, where $e,e_1$, and $e_2$ are primitive elliptic curves with $e \cdot e_1=e \cdot e_2=e_1 \cdot e_2=1$ and $k \ge 1$  (genus $g=2k+2$).
\end{enumerate}
(in the first two cases we have set $k \ge 3$ and $k \ge 2$ to prevent from falling in the cases (3) and (1) above). Let us do the  case (c) with $k=1$, which is the well known realization of an Enriques surface as the normalization of a sextic surface $R$ in $\P^3$ that passes doubly through the edges $l_1, \dots, l_6$ of a tetrahedron \cite{Dolgachev16}. The edges of the tetrahedron are the images of the elliptic curves $e, e', e_1, e_1', e_2, e_2'$ (indeed, the linear system $|C|$ restricts to a $g^1_2$ on each of these elliptic curves).  In addition to the double lines, the surface $R$ has $4$ pinch points on each of the edges of the tetrahedron (the ramification points of the $g^1_2$'s on the elliptic curves) and $4$ triple points at the $4$ vertices of the tetrahedron. In particular, the preimage of a general point on one of the lines $l_i$ consists in $2$ points, the preimage of the pinch points consists in one single point, and the preimage of the triple points consists in $3$ points. If $H$ is a general hyperplane section through one of the pinch points, then it acquires a cusp. However, since $T \to R$ is the normalization morphism, we can see that in this case $\psi^{-1}(H)$  is smooth. Indeed, $T$ can be locally identified with the proper transform of $R$ under the blow up of $\P^3$ along the double line, and it is clear that this blow up normalizes a general cuspidal curve passing through a pinch point. A general hyperplane section through one of the triple point will be a curve with a triple point with three distinct branches which are separated under the map $\psi$. It follows that curves in $|C|$ that have worst singularities than one simple node appear in codimension two.

The other cases can be dealt with analogously: using  Lemma \ref{birational map} below and Definition 1.43 (and discussion thereafter) of \cite{Kollar-Kovacs}, $R$ has two double lines (and is generically normal crossing along them)  and pinch points. Hence, a hyperplane section $\Gamma$ of $R$ is singular wherever it is tangent to the smooth locus of $R$ and is also singular along its intersection with the two double lines. A tangent hyperplane section that does not contain the double locus, will thus be normalized under the morphism $T \to R$. 
 It follows that the discriminant locus of these linear system is equal to the closure of the locus of hyperplane sections of $R$ that are tangent along the smooth locus of $R$ which, again using  \cite[\S 2.1 II]{Voisin2}, is irreducible.
 \end{proof}
 
 \begin{cor} \label{discriminant locus irr}
Let $|C|$ be as in case $(2)$ of Proposition \ref{Verra} above. Then the discriminant locus of $|C|$ is irreducible. Moreover, if $|C|$ is as in case $(1)$ of Proposition \ref{Verra}, the number of irreducible components of the discriminant locus is equal to $37$.
\end{cor}
\begin{proof}
The second statement follows from discussion of Verra's result in the proof of Proposition \ref{Verra}. As for the second statement, we can argue as follows. First recall that the closure of locus of hyperplane sections that are tangent to a smooth quasi--projective variety (as is the smooth locus of $R$) is irreducible. Second, notice that for any hyperplane $H \subset \P^{g-1}$ that does not contain the two lines of the non--normal locus of $R$, the curve $H\cap R$ has only nodes or cusps along the two lines (depending if it meets a line at a regular double point or at a pinch point) and hence it is normalized under the map $ T \to R$ (notice that the hyperplane sections that contain the lines appear in higher codimension).
\end{proof}


We remark that since in the case (3) the class of $C$ is divisible by $2$, this is not a case we will consider (a necessary condition for the assumptions of Theorem \ref{thm smooth} to hold is that the class of $|C|$ is not divisible by $2$).

\begin{lemma} \label{cuspidi}
Let $|C|$, $\psi: T \to \P^2$, and $\Gamma$ be as in case (1)  above, and let $\gamma \in \Gamma$ be a cuspidal point. Then for a general line $\ell$ through $\gamma$, the curve $C=\psi^{-1}(\ell)$ has at worst one node.
\end{lemma}
\begin{proof}  Since $\ell$ is a  general line through $\gamma$, we may assume that it is not tangent to $\Gamma$ away from $\gamma$ and hence that $C$ is smooth away from $\psi^{-1}( \gamma)$.
Let $\wt C \to C$ be the normalization of $C$ and let us consider the induced morphism $ \wt \psi: \wt C \to C \to \ell$. Suppose that $C$ has a cusp over $\gamma$, so that $\wt C$ has genus $3-1=2$ and  $\wt \psi$ has a ramification point over $\gamma$. Applying Riemann--Hurwitz to $\wt \psi$, we can compute the ramification $r$ of $\wt \psi$: $r=2 \deg \psi + 2g-2= 2 \cdot 4+2=10$. Since $\ell$ already meets the ramification curve $\Gamma$ in $10$  points other than $\gamma$, $\wt \psi$ cannot ramify over $\gamma$. Hence, if $C$ has a double point over $\gamma$ it must be a node. The case where $\wt C$ has worst singularities is dealt with analogously, using Riemann--Hurwitz to compute the ramification divisor and finding a contradiction on the number of ramification points of $\wt \psi$ outside of $\gamma$.
\end{proof}

\begin{lemma} \label{birational map}
Let $|C|$ and $\psi: T \to R \subset \P^{g-1}$ be as in (a), (b) or  in (c) with $k \ge 2$ above (so that $g \ge 5$). Then $\psi$ is an isomorphism outside of the two elliptic curves $e$ and $e'$, which are mapped $2:1$ onto two double lines in $R$. 
\end{lemma}
\begin{proof}
It is easy to check that the restriction of $|\mc O(C)|$ to $ {e}$ and $e'$ is a $g^1_2$.  Also notice that since two sublinear systems $|I_e(C)|=|\mc O(C-e)$ and $|I_{e'}(C)|=|\mc O(C-e')$ are different subspaces of $|\mc O(C)|$, it follows that $|\mc O(C)|$ separates the two curves $e$ and $e'$. In the case when $C^2 \ge 10$ (which corresponds precisely to $k \ge3$ in cases  $a)$ and $b)$ above and to $k \ge 2$ in case $c)$), we can use Reider's theorem together with our assumptions on $T$ and on $|C|$, to conclude that $\psi$ separates points and tangent directions outside of an effective curve whose components $E$ satisfy $E^2=0$ and $E\cdot C=2$. Since the component of such a curve have to be equal to $e$ or to $e'$, this solves the question in the case when $C^2 \ge 10$. We thank the referee for suggesting the use of Reider's theorem.
We are thus only left to consider the case $b)$, with $k=2$. 
This corresponds to the case when $|C|=|2e+e_1|$ and $e$ and $e_1$ are primitive elliptic curves with $e \cdot e_1=2$. As mentioned on page 278 of \cite{Cossec-Dolgachev} the projective model associated to this linear system is a non--normal octic surface in $\P^4$ with two double lines (images of $e$ and $e'$). We include a proof of the fact that, for general $T$, such an octic surface is smooth away from the two lines. We need to show that for every length $2$ point $z $ on $ T$ that is scheme theoretically not contained in $ \{e \cup e'\}$, there is a surjective map
\[
\alpha_z: H^0(\mc O(C)) \to H^0(\mc O_z(C)). 
\]
We will prove this with the help of the linear systems $D:=C-e=e+e_1$ and $D'=C-e'=e'+e_1$. Notice that,  as in case $(1)$ of Proposition \ref{Verra}, $|D|$ defines a degree $4$ morphism $\phi_D: T \to \P^2$ (cf. \cite[Thm 4.6.3]{Cossec-Dolgachev}). Given $z$ as above, then either $z$ is scheme-theoretically contained in a fiber  of $\phi_D$, or it is not. Suppose it is, so that $ z \subset \phi_D^{-1}(p)$, for some $p \in \P^2$. Then $|I_{z|T}(D)|=\phi_D^*|I_{p|\P^2}(1)|$ is a pencil and hence there is an irreducible curve $D_z \in |D|$ containing $z$ (here we denote by $I_{X|Y}$ the ideal sheaf of a closed subscheme $X \subset Y$). We claim that $H^1(D_z,I_{z|D_z}(C))=H^0(D_z,I_{z|D_z}^\vee(-C)\otimes \omega_{D_z})=0$. This immediately implies that $H^0(D_z, \mc O_{D_z}(C)) \to H^0(\mc O_z(C))$ is surjective, which,  since $H^1(T, \mc O(C-D_z))=H^1(T, \mc O(e))=0$,  implies that $\alpha_z$ is also surjective. To prove the claim, suppose by contradiction that there is a non zero section $\sigma \in H^0(D_z,I_{z|D_z}^\vee (-C)\otimes \omega_{D_z})= H^0(D_z,I_{z|D_z}^\vee(-e'))$. Then $\sigma $ induces an injective morphism $\mc O_{D_z} \to I_{z|D_z}^\vee(-e')$ which we can dualize (notice that $ I_{z|D_z}$ is reflexive because it is a torsion free sheaf on a locally planar curve)
 to get an injection 
\[
\sigma^\vee: I_{z|D_z}(e') \to \mc O_{D_z}.
\]
Using the fact that $H^1(e_1)=0$, we see that $|\mc O(e') \res{D_z}|=\{e'\cap D_z\}$. The existence of a non--zero $\sigma^\vee$ as above implies hence that $z\supset e'\cap D_z$. Since the length of $z$ and of $e'\cap D_z$ are both equal to $2$, this implies that $z= e'\cap D_z$, which contradicts the fact that $ z$ is  scheme theoretically not contained in $e' \cup e$.

Let us now suppose that $z $ is not contained in a fiber of $\phi_D$. Then the morphism $ \beta_z: H^0(\mc O(D)) \to H^0(\mc O_z(D))$ is surjective. Consider the morphisms
\[
\gamma: H^0(\mc O(D)) \to H^0(\mc O(C)), \quad \text{and} \quad \gamma_z:  H^0(\mc O_z(D))=H^0(\mc O_z(C-e)) \to H^0(\mc O_z(C)).
\]
Since $ \gamma_z \beta_z=\alpha_z \gamma$ and $\beta_z$ is surjective, we have $\im (\gamma_z) \subset \im (\alpha_z)$. Notice that $\gamma_z$ vanishes along $e\cap z$. There are three cases. The first is when $z \cap e= \emptyset$ so that  $\gamma_z$ is an isomorphism and hence $\alpha_z$ is surjective. The second is when $z$ is a length two point supported on $e$ (but not scheme theoretically contained in $e$). Then, under the identification $\mc O_z(C) \cong \C[\epsilon]\slash (\epsilon^2)$, we see that $\im(\gamma_z)$ is the maximal ideal $(\epsilon)$ and to conclude that $\alpha_z$ is surjective we only need to notice that $|C|$ is base point free and hence that there is a section not vanishing on the support of $z$. The third case is when $z=z_1 \cup z_2$ with $z_1 \in e$ and $z_2 \notin e$. Then $\im(\gamma_z) =\mc O_{z_2}$. In this case, we can use the linear system $D'=C-e'=e'+e_1$ and consider instead the morphism $\gamma'_z:   H^0(\mc O_z(D'))=H^0(\mc O_z(C-e')) \to H^0(\mc O_z(C))$. The same reasoning as above, together with the fact that $ z_1 \notin e'$, shows that  we  have $\im(\gamma'_z)\supset \mc O_{z_1} $. Since $\im(\alpha_z) \supset \im(\gamma'_z)$, this shows that also in the third case $\alpha_z$ is surjective and concludes the proof.

\end{proof}

\begin{prop} \label{discriminant locus}  
(1) Let $|C|$ be a hyperelliptic of genus $g \ge 3$. Then $\Delta \subset |C|$ is the union of four irreducible components $\Delta_1, \Delta_2, \Delta_3, \Delta_4$. The general point of the first two components parametrizes curves that are the union of two smooth curves meeting transversally in one point, the general point of the third component parametrizes curves that are union of two smooth components meeting transversally in two points, and the general point of the fourth component parametrizes singular, but irreducible, curves. Moreover, the general curve parametrized by this component has only one node. \\
(2) If $g(C)=2$, then $|C|$ is a pencil which for general $T$ has exactly $18$ singular members, $16$ of which are irreducible with one node and $2$ of which are reducible, consisting of two elliptic curve meeting transversely in one point.
\end{prop}
\begin{proof}
(1) It is clear that the two hyperplanes,
\be \label{two hyperplane components}
\Delta_1:=\{e_1\}\times |(n-1)e_1+e_2|,\quad \text{ and } \quad \Delta_2:=\{e_1'\}\times |(n-1)e_1+e_2'|,
\ee
constitute two components of the discriminant locus, and also that they parametrizes curves of the form $e_1\cup \Gamma$, with $\Gamma$ a curve in $ |(n-1)e_1+e_2|$ (resp. $e_1'\cup \Gamma$, with $\Gamma \in |(n-1)e_1+e'_2|$).  Since the genus of $\Gamma$ is $\ge 2$, the general curve in these linear systems is smooth by Corollary \ref{smooth general member}.  
For the third irreducible component, consider the natural map
\[
\begin{aligned}
\psi : \P^1 \times \P^{n-2}=|2e_1| \times |(n-2)e_1 \times e_2| & \longrightarrow \Delta \subset \P^n \\
 (C_1, C_2) & \longmapsto C_1+C_2
\end{aligned}
\]
which is finite and birational. In particular, the image of $\psi$ defines a component of $\Delta$ which we denote by $\Delta_3$.  The general curve parametrized by this component is therefore the union of a smooth curve in $|2e_1| $ and of a smooth curve in $ |(n-2)e_1 \times e_2| $, which generically meet transversally in two distinct points.


We are left with proving that the remaining part $\Delta_4$ of the discriminant is irreducible and that the general curve parametrized by it is irreducible with one single node. By definition of hyperelliptic linear system, and the fact that $T$ contains no rational curves, the rational map $\phi_{|C|}$ associated to the linear system maps $T$ generically $2:1$ onto a degree $n-1$ smooth rational surface $R \subset \P^{n}$. We know recall some geometry of $\varphi_{|C|}$, following \cite[Thm 4.5.2]{Cossec-Dolgachev}.
We already saw that the linear system has two base points, which were described in Lemma \ref{base points}.
 Since the degree one linear systems  induced by restricting $|C|$ to $e_1$ and $e_1'$ both have a base point,  $\phi_{|C|}$ contracts these two curves to two distinct points, denoted by $P$ and  $Q$. By Lemma \ref{base points}, when $n$ is even then both base points of $|C|$ lie on $e_2$  so that when $n=2$  the degree two linear systems $|C| \res{e_2}$ is the trivial and the curve $e_2$ gets also contracted.
Let $T' \to T$ be the blow up of the two base points of $|C|$. We get a generically $2:1$ morphism $ T' \to R$ which contracts the proper transforms of $e_1 $ and $ e_1'$ (and also $e_2$ if $n=2$).

The ramification curve of $\phi_{|C|}$ is described in Theorem 4.5.2 of \cite{Cossec-Dolgachev}.
It consists of the union of two lines $\ell_1$ and $\ell_2$ (the images of the exceptional divisor of the blow up $T' \to T$) and of an irreducible curve $B \subset R$. The irreducible curve $B$ has two tacnodes in $P$ and $Q$ and is otherwise non singular ($T$ contains no rational curves), except in the case $n=2$ where it is has a simple node at the intersection $O$ of the two lines.

A curve in $|C|$ is singular in the following three cases.
If it covers a singular (hence reducible) hyperplane section of $R$,  if it covers a smooth curve that is tangent to the ramification curve, if its image contains one of the two points $P$ and $Q$ or, when $n=2$, if it covers a line passing by the intersection $O$ of $\ell_1$ and $\ell_2$.

The preimages, under $\phi_{|C|}$, of the hyperplane sections through $P$ and $Q$ are the curves in $\Delta_1$ and $\Delta_2$, respectively. 
The set of hyperplane sections of $R$ that are tangent to $\ell_1$ (resp. $\ell_2$) contain  $\ell_1$ (resp. $\ell_2$) and hence they form a set of codimension $2$.  For $n=2$, we also have to consider the set of hyperplanes through $O$, which is nothing but $\Delta_3$, namely, the sublinear system  $e_2+|2e_1|$  of curves containing $e_2$. While for $n \ge 3$ the component $\Delta_3$ corresponds to the curves covering the reducible hyperplane sections of $R$.

Finally, we observe that the closure of the set of hyperplane sections that are tangent to $B$ at smooth points is irreducible, since it is dominated by a $\P^{n-2}$--bundle over the smooth locus of $B$. Moreover, generically it parametrizes tangent curves that are tangent but not bi--tangent, so that the corresponding curve in $|C|$ has one simple node.

(2) A genus $2$ linear system has two simple base points $p$ and $q$, so that if $\mc C \to |C|$ denotes the universal family, we have $\chi_{top}(\mc C)=\chi_{top}(BL_{p,q}T)=14$. We use this to count the number of singular curves in $|C|$. By Proposition \ref{hyperelliptic enriques}, $|C|=|e_1+e_2|$ for two primitive elliptic curves $e_1$ and $e_2$, with $e_1 \cdot e_2=1$. It follows that that there are exactly two reducible curves in $|C|$, namely $e_1+e_2$ and $e'_1+e'_2$.  By \S 8.1.4 (i) of \cite{Cossec}, the linear system $|C|$ is the pullback under a degree two map $\psi: T \dashrightarrow |C| \times |C'|=\P^1 \times \P^1 \subset \P^3$ of one of the two rulings of the quadric. The map $\psi$, which is defined away from the four intersection points $e_i \cap e'_j$ ramifies over the union of a square of lines (the images of the four exceptional divisors of the blow up of $T$ at  $e_i \cap e'_j$ and of a degree $(4,4)$ curve $B$ that has a simple node at each edge of the square. Counting parameters we can see that the general Enriques surface can be constructed in this way, and that generically $B$ will have no other singularities. Moreover, we can also assume that  no line in the ruling is bitangent (for more details see \cite{thesis}). Since the singular, but irreducible, members of $|C|$ arise from lines that are tangent to $B$, we see that generically they all have one simple node and no other singularity. Hence all the singular curves in $|C|$ have Euler number equal to $-1$. If $N$ denotes the number of singular fibers, we have $14=-2(2-N)-N$,  and there are exactly $18$ singular fibers, two of which are reducible.
\end{proof}

\section{Smoothness and first properties of $N$} \label{fibration in jacobians}


Let $\chi$ be a non-zero integer and set
\be \label{w e v}
w=(0,[C], \chi), \,\,\,\, \text{ and } \,\, v=f^*w=(0,[D], 2\chi).
\ee
Let $H$ denote a polarization on $T$, and set
\[
A=f^*H.
\]
The Hilbert scheme $\mc H_{c_1(w)}$ has two components: $|C|$ and $|C'|$. Without loss of generality we can consider only one of them and set
\be \label{N}
N=M_{w,H}(T, |C|), \,\,\,\, \text{ and } \,\, M=M_{v,A}(S).
\ee


In order to study $N$, we will look at the natural pullback morphism from $N$ to $M$, which to a sheaf $F$ on $T$ with $v(F)=w$, associates the sheaf $f^*F$ on $S$ with $v(f^* F)=v$. We start with a few well known lemmas

\begin{lemma}[\cite{Gieseker}] \label{pullback stable}
Let $G$ be a pure dimension one sheaf on $T$ and let $H$ be an ample line bundle on $T$ and set $A=f^*H$. If $G$ is $H$-semistable, then $f^*G$ is $A$-semistable on $Y$.
\end{lemma}
\begin{proof}
Since $\omega_T$ is numerically trivial, tensoring a sheaf by $\omega_T$ does not change the numerical invariants of the sheaf itself. It follows that the operation of tensoring by $\omega_T$ preserves not only the slope, but also stability and semi-stability with respect to any line bundle. Let $E \subset f^* G$ be a subsheaf. Using the projection formula for $f$, we see  that $f_* E$ is a subsheaf of the $H$-semistable sheaf $G\oplus (G\otimes \omega_T) $.
Clearly,  $\chi(E)=\chi (f_* E)$. Moreover, since $E$ is pure of dimension one, $f_* c_1(E)=c_1(f_* E)$.  Then $c_1(E)\cdot H=f_*c_1(E)\cdot H$, so that $\mu_H(E)=\mu_{H}(f_*E)$. Since $\mu_{H}(G)=\mu_{A}(G\oplus (G\otimes \omega_T) )$, the lemma is proved.
\end{proof}

\begin{lemma}  \label{F and F'}
Let $E$ and $G$ be two non isomorphic $H$-stable sheaves on $T$. Suppose that $f^*E \cong f^*G$. Then
$$G\cong E \otimes\omega_T.$$
\end{lemma}
\begin{proof}

If $f^*E \cong f^*G$, then we also have an isomorphism $E\otimes (\mc O_T \otimes \omega_T) = f_*f^*E \cong f_*f^*G=G\otimes (\mc O_T \otimes \omega_T) $. Since all maps from $E$ to $G$ are trivial, it follows that  the composition $E \to E\otimes f_* \mc O_S \to G\otimes \omega_T$ has to be non-zero. However, since $E$ and $G\otimes \omega_T$ are stable of the same reduced Hilbert polynomial, we have
$$E \cong G\otimes \omega_T.$$
\end{proof}


\begin{lemma} \cite{Takemoto73} \label{takemoto}
Let $G$ be a sheaf on $T$. If $G \cong G \otimes \omega_T$ then $f^* G$ is not simple. In particular, it cannot be stable.
\end{lemma}
\begin{proof}
We have
\be
\begin{aligned}
\Hom(f^*G, f^*G)&=\Hom(G, f_* f^*G)=\\
&=\Hom(G, G)\oplus \Hom(G, G\otimes  \omega_T)\\
\end{aligned}
\ee
Since, by assumption, $\Hom(G, G\otimes  \omega_T)$ is at least one-dimensional the Lemma is proved.
\end{proof}

By Lemma \ref{pullback stable} the pullback map
\be \label{pullback morphism}
\begin{aligned}
\Phi: N=M_{w,H}(T, |C|)&\longrightarrow M=M_{v,A}(S),\\
[G] &\longmapsto [f^* G]\\
\end{aligned}
\ee
is a regular morphism.

\begin{lemma}[\cite{Kim98}]\label{pullback generically 2 to 1} 
The pullback morphism $\Phi: N \to M$ is generically $2:1$.
\end{lemma}
\begin{proof}
By Lemma \ref{F and F'} the morphism $\Phi$ is of degree $ \le 2$ so that we only need to prove that, for a sheaf $G$ corresponding to a general point in $N$, the sheaf $G$ is not isomorphic to $G \otimes \omega_T$. By Corollary \ref{smooth general member} the general member of $|C|$ is smooth. By Lemma \ref{existence stable sheaves}, if the Fitting support of $G$ is a smooth curve, then $f^* G$ is $H$-stable so that we may use Lemma \ref{takemoto} and conclude that $\Phi$ is generically $2:1$.
\end{proof}

As we remarked above, tensoring by $\omega_T$ preserves stability hence the involution
\be \label{involution on N}
\begin{aligned}
\epsilon: N &\rightarrow N,\\
G &\mapsto G \otimes \omega_T,
\end{aligned}
\ee
is well defined. It clearly commutes with $\Phi$.

\begin{lemma} \label{N smooth at stable points}
Let $G$ be an $H$-semistable sheaf such $f^*G$ is $A$-stable. Then $G$ is $H$-stable and $N$ is smooth at $[G]$ of dimension $2g-1$.
\end{lemma}
\begin{proof}
The first statement is clear. The obstructions to deforming a $G$ on $T$ lie in $\Ext^2(G,G)$ which is dual, by Serre duality, to
\[
\Hom(G, G \otimes \omega_T),
\]
However, this space is zero by Lemma \ref{takemoto} as we are assuming $f^*G$ to be simple.
\end{proof}

Since $\iota^* f^* G=f^* G$ the image of $\Phi$ is contained in the closure of the fixed locus of the birational involution
\be \label{iota star}
\begin{aligned}
\iota^*: M=M_{v,A}(S) & \dasharrow M_{v,A}(S)=M.\\
F &\longmapsto \iota^* F
\end{aligned}
\ee
Notice that this involution is regular on an open subset containing sheaves with irreducible support.

\begin{lemma} \label{involution anti-symplectic}
The involution $\iota^*: M \dasharrow M$ is anti-symplectic, i.e. if $\sigma$ denotes the symplectic form on the smooth locus of $M$, then $\iota^* \sigma=-\sigma$. Moreover, the fibration $\pi: M \to |D|$ is equivariant with respect to the involution $\iota^*$ defined above.
\end{lemma}
\begin{proof} Let $F$ be an $A$-stable sheaf corresponding to a point $[F]$ in $M$.
By functoriality of the cup product and of the trace map, the following diagram is commutative,
\be
\xymatrix{
\Ext^1(F, F) \times \Ext^1(F, F) \ar[d]_{\iota^*} \ar[r] ^{\phantom{ghim}\cup} & \Ext^2(F, F) \ar[r]^\tr \ar[d]_{\iota^*} &  H^2(S, {\mc O}_S)\ar[d]^{\iota^*}\\
\Ext^1(\iota^*F, \iota^*F) \times \Ext^1(\iota^*F, \iota^* F) \ar[r] ^{\phantom{ffjgjgjf} \cup} &\Ext^2(\iota^*F,\iota^* F) \ar[r]^\tr &  H^2(S, {\mc O}_S).
}
\ee
Hence by Mukai's description of $\sigma$ (cf. (\ref{symplectic form})),  to prove the Lemma we only need to prove that the identification $H^2(S, {\mc O}_S)\cong \C$ changes sign if we compose it with $\iota^*$. This follows from the fact that, since $\iota$ is an anti-symplectic involution on $S$, the identification $H^0(S, \omega_S)=\C \sigma \cong \C$ changes sign once we compose it with $\iota^*$. The second statement follows from the definitions of $\iota^*$ and $\pi$.
\end{proof}

\begin{lemma} \label{fix iota smooth}
If $H$ is $\iota^*$-invariant, then the involution (\ref{iota star}) is regular.
Let $Z $ be any component of $\Fix(\iota^*) \subset M$, then $Z \cap(M\setminus \Sing(M))$ is smooth. Moreover, if $Z\cap(M\setminus \Sing(M))$ is non empty, then $Z$ is an isotropic subvariety of $M$.
\end{lemma}
\begin{proof}
The first statement is clear, since if $F$ is $H$-stable then $\iota^*F$ is $\iota^*H$-stable. The second statement follows from the well known fact that the fixed locus of the action of a finite group on a smooth variety is smooth. As for the third statement, it is an immediate consequence of Lemma \ref{involution anti-symplectic}.
\end{proof}

\begin{lemma} \cite{Kim98} \label{simple and invariant}
Let $F$ be a pure sheaf of dimension one on $X$ and assume that it is $\iota^*$-invariant. If $F$ is simple, then
\[
F=f^*(G),
\]
for some sheaf $G$ on $Y$.
\end{lemma}
\begin{proof}
Since $Y$ is a quotient of $X$ by a $\Z/(2)$ action, the descent data translates into the existence of a morphism $\varphi: \iota^* F \to F$, such that the following diagram is commutative
\[
\xymatrix{
\iota^* \iota^* F \ar@{=}[d] \ar[r]^{\iota^*\varphi } & \iota^* F \ar[r]^\varphi & F \\
F \ar[urr]_\id &&
}
\]
Since $F$ is simple, this can always be achieved by multiplying any given isomorphism $\iota^* F \to F$ by a suitable scalar.
\end{proof}

Now let 
\be \label{Y}
Y:=Y_{v,A} \subset \Fix(\iota^*),
\ee
be component of $\Fix(\iota^*)$ containing $\Fix($. The lemma above says that the restriction
\[
\Phi: N \to Y \subset M,
\]
(which by abuse of notation we still denote by $\Phi$) is surjective.

Before stating the main result of this section, we recall the following definition.

\begin{defin} \cite[Def. 3.8]{Yoshioka-FM} \cite[Theorem-Definition 2.4]{AS} \label{definition of walls}
Let $v$ be a primitive Mukai vector. A polarization $H$ is called $ v$--generic if any $H$--semistable sheaf with Mukai vector $v$ is actually $H$--stable.
\end{defin}

By \S 1.4 of \cite{Yoshioka},  if $v=(0, D, \chi)$ is primitive Mukai vector and $\chi \neq 0$, then the locus of $v$--generic polarizations is non--empty. More precisely, this locus is equal to the complement of a finite union of real codimension one subset of $\Amp_\R(T)$. A \emph{wall of $v$} is defined to be an irreducible component of the complement of the locus of $v$--generic polarizations. In \cite[Prop 2.5]{AS} explicit equations for the walls are given for primitive Mukai vectors of pure dimension one (notice that the set of walls could, a priori, be a proper subset of the linear subspaces appearing in \S 1.4 of \cite{Yoshioka}).  It is not hard to see that if $\chi=0$, then the set of $v$--generic polarizations can be empty.

\begin{thm} \label{thm smooth}
Let $C$ be a curve of genus $g \ge 2$ on an Enriques surface $T$. Let $\chi$ be a non-zero integer, set $w=(0,[C], \chi)$ and $v=(0,[D], 2\chi)$, where $D=f^{-1}(C)$. Assume that $v$ is primitive, and let $A$ be an ample line bundle on $T$ such that $A=f^*H$ is $v$-generic. The moduli space
\[
N=M_{w,H}(T, |C|),
\]
is a smooth, irreducible, projective variety of dimension $2g-1$, and it admits an \'etale double cover onto the Lagrangian subvariety
$ Y_{v,A}(S) \subset M_{v,A}(S)$.
\end{thm}
%
%
\begin{proof}
The smoothness follows from Lemmas \ref{N smooth at stable points}, \ref{fix iota smooth} and \ref{simple and invariant} above.
The fact that the pullback morphism is unramified follows from Lemma \ref{F and F'}.
Notice that since $\dim M_{v,A}(S)=2h$ and $h=2g-1$, the isotropic subvariety $Y_{v,A}(S) \subset M_{v,A}(S)$ is indeed Lagrangian. The fact that $N$ is irreducible follows from \cite[Thm 0.2]{yoshioka-enriques}.

\end{proof}

\begin{rem}
In the rest of the paper we will usually refer to $N$ without mention of the dependency on $w$ and $H$. The phrase `` let $N$ be as in Theorem \ref{thm smooth}'' will mean ``let $w$ and $A$ be as in Theorem \ref{thm smooth} and set $N= M_{w,A}(T, |C|)$''. If we refer to $N$ as a relative compactified Jacobian of specific degree $d$, then it means that we have chosen $\chi=d-g+1$ in $w=(0, C, \chi)$.
\end{rem}

\begin{rem}
In the theorem above we have asked $\chi \neq 0$. This is because otherwise the canonical bundle of a reducible curve would be strictly semistable. This condition appears also in \cite{Yoshioka}.
\end{rem}

One can also verify directly that if $G \ncong G \otimes \omega_T$, then the differential
\[
d\Phi: \Ext^1_T(G,G) \to \Ext^1_S(f^*G,f^*G),
\]
is injective.


Notice that if $w$ is primitive, then so is $v$ as soon as $C$ is not divisible by $2$ in $\NS(T)$. Moreover, if $T$ is general, then the general $H$ in $\Amp(T)$ will be such that $f^*H$ is $v$-generic.

\begin{rem}
If the assumptions of the proposition are not satisfied, the singular locus of $M$ and of $N$ may be non-empty. For vector bundles, this singular locus has been described by Kim in \cite{Kim98}.
\end{rem}

\subsection{On the support morphism}

Regarding the relative compactified Jacobian over the locus of reduced curves, we have the following result of Melo, Rapagnetta and Viviani

\begin{prop}[\cite{Melo-Rapagnetta-Viviani14}] \label{flat on reduced locus}
The restriction 
\[
N_{V} \to V,
\]
of $\nu$ to the open locus $V$ of reduced curves is equidimensional. In particular, if $N$ is smooth then $N_V \to V$ is flat. 
\end{prop}
\begin{proof}
This follows from the cited result Proposition \ref{MRV on compactified Jac}.
\end{proof}

Problems, however, may arise when dealing with non-reduced curves. In general, the Simpson moduli spaces of sheaves on a non-reduced curve may have higher dimensional components, as the following example shows, and are not well understood.

\begin{example}\label{jesse}{\rm Consider a smooth curve $\Gamma'$ of genus $\gamma' \ge 2$, and let $\Gamma$ denote the scheme obtained by considering a non-reduced double structure on $\Gamma'$. Let $\gamma$ be the genus\footnote{By genus we mean the \emph{arithmetic} genus of $\Gamma$, i.e., the integer $\gamma$ defined by $\chi(O_\Gamma)=1-\gamma$.} of $\Gamma$. It was shown by Chen and Kass in \cite{Chen_Kass11}, that all the components of the Simpson moduli space have dimension $\gamma$ except, possibly, a $(4\gamma'-3)$-dimensional component, which exists when $4\gamma'-3 \ge \gamma$. This component parametrizes rank 2 semistable sheaves on $\Gamma'$.
Suppose now that $\Gamma$ and $\Gamma'$ are contained in a smooth surface $X$, so that the scheme structure defining $\Gamma$ is the one induced by the ideal sheaf ${\mc O}(-2\Gamma')$. By the adjunction formula,
\[
\gamma=4\gamma'-3-\deg {\omega_X}_{|\Gamma'}.
\]
so the dimension of the Simpson moduli space does not jump when
\be \label{canonical bundle numerically trivial}
\deg {\omega_X}_{|\Gamma'}\le 0.
\ee
}
\end{example}

In particular, as soon as the canonical bundle of $X$ is numerically trivial ( K3, abelian,  Enriques or bi-elliptic surfaces) then (\ref{canonical bundle numerically trivial}) is satisfied for any curve contained in $X$.

\begin{conjecture} \label{expectation}
Let $(X,H)$ be a smooth projective surface and let $C \subset X$ be a curve of arithmetic genus $g$. If $\deg {\omega_X}_{|\Gamma}\le 0$ for every sub curve $\Gamma \subset C$, then any component of the Simpson moduli space of pure dimension one sheaves with support equal to $C$ is $g$-dimensional.
\end{conjecture}

Evidence for this conjecture is given by the following examples.

Suppose $(X, H)$ is a polarized K3 or abelian surface and let $v=(0, D, \chi)$ be a primitive Mukai vector with $\chi \neq 0$. Matsushita proved in \cite{Matsushita1} that the support morphism
\[
\pi: M_{v,H} \to |D|,
\]
is equidimensional. The proof, however, relies on the existence of a symplectic structure on these moduli spaces and cannot be applied to moduli spaces of sheaves on other surfaces. Indeed, using Koll\'ar's theorem on the torsion freeness of the higher direct images of the structure sheaf, Matsushita proves that \emph{ every } fiber of $\pi$, and not just the general one, is Lagrangian and hence of dimension equal to $\dim M_{v,H}(X)/2$. Here, by Lagrangian, we mean that the pullback of the symplectic form to any resolution of a fiber, considered with its reduced induced structure, vanishes identically.

Another example where the conjecture holds true is the Hitchin system for the group $GL(r)$, which can be thought of as the relative compactified Jacobian of a linear system on the ruled surface $X$ associated to the canonical bundle of a curve. The spectral curves are multi sections of the ruling and hence satisfy $\omega_S\cdot C =0$. Also in this case, the proof of flatness comes from the existence of a symplectic form with respect to which the Jacobian fibration is Lagrangian \cite{Laumon}. The condition $\deg {\omega_X}_{|\Gamma}\le 0$ appears also in the recent paper \cite{CL}, which provides further evidence for the conjecture.

However, it is natural to expect that the dimension of the fibers of the support map should not depend on the existence of a symplectic structure but only on discrete invariants such as the rank of the sheaves and the arithmetic genera of their supports.

The last example is provided by Del Pezzo surfaces. In \cite{LePotier} Le Potier shows that for $\P^2$ the Picard group of the moduli spaces has two generators: the pullback of the hyperplane section under the support morphism and the determinant line bundle. Looking into the proof, however, one realizes that for any generically polarized Fano surface the fibers of the support morphism are not too big. More specifically, one can use Lemmas 3.2 and 3.3 of loc. cit., and the fact that in this setting one can choose the Quot scheme so that it is a principal bundle over the moduli space, to show that the locus of sheaves supported on non-reduced curves has codimension greater or equal to two.

Since we were not yet able to prove Conjecture \ref{expectation} for Enriques surfaces, we will need the following assumption when computing the fundamental group and the second Betti number in Sections \ref{fundamental group} and \ref{second betti number}.

\begin{assumption}\label{assumption}
As above, let $N_V \to V$ be the restriction of $\nu$ to the locus $V \subset |C|$ of reduced curves. The linear system $|C|$ is such that 
\[
\codim(N \setminus N_V, N) \ge 2.
\]
\end{assumption}

Since $\codim (V, |C|)\ge 2$ this assumption is equivalent to asking that there are no irreducible components of $N_\Delta$ which map to codimension $\ge 2$ subsets of $\Delta$. 

In some cases of low genus, where the curves of the linear system do not degenerate too much, one can show directly that the relative compactified Jacobian is equidimensional.  For example, if there are no non-reduced curves, or if all the non-reduced curves have at worst a double structure one can use Propostion \ref{flat on reduced locus} and Example \ref{jesse}. Some examples of linear systems all of whose members are reduced are
\be \label{low genus linear systems}
\begin{aligned}
&\vert e_1+e_2 \vert, \,\,\, \text{with}  \,\,\, e_1\cdot e_2=1,& g(C)= 2, & \,\, \dim N=3, \\
&|e+f|  \,\,\, \text{with}  \,\,\, e\cdot f=2,&g(C)=3, & \,\, \dim N=5,\\
&\vert e_1+e_2+e_3 \vert  \,\,\, \text{with}  \,\,\, e_i\cdot e_j=1 \,\, \text{for}\,\, i\neq j,& g(C)=4, & \,\, \dim N=7.
\end{aligned}
\ee
where $e_1$, $e_2$, $e_3$, $e$ and $f$ are primitive elliptic curves. 

More generally, as Yoshioka has pointed out to me, this assumption is satisfied whenever $|C|$ is primitive:

\begin{prop} \cite{yoshioka-enriques} \label{Yoshioka assumption} Let $|C|$ be a primitive linear system on a general Enriques surface $T$, and $N$ be as in Theorem \ref{thm smooth} (i.e. let $w$ and $H$ be as in Theorem \ref{thm smooth} and set $N= M_{w,H}(T, |C|)$ ). Then Assumption \ref{assumption} is satisfied.
\end{prop}
\begin{proof} This is Proposition 4.4 of \cite{yoshioka-enriques}.
\end{proof}

\begin{cor} \label{assumption hyperelliptic}
Let $|C|$ be a hyperelliptic linear system on a general Enriques surface $T$, and let $N$ be as in Theorem \ref{thm smooth}. Then Assumption \ref{assumption} is satisfied.
\end{cor}
\begin{proof}
By $(3)$ Proposition \ref{hyperelliptic enriques}, a hyperelliptic linear system is primitive and hence we may use Proposition \ref{Yoshioka assumption}
\end{proof}

\section{The fundamental group} \label{fundamental group}

This section is devoted to computing  the fundamental group of the relative compactified Jacobian variety $N$ constructed in Section \ref{fibration in jacobians}. We show that there is a surjection $\Z/(2) \twoheadrightarrow \pi_1(N)$ which, under Assumption \ref{assumption} is actually an isomorphism. Under this assumption, we can also identify the universal covering space, which can be described using the norm map. At the end of the section, we also prove some results on the vanishing cycles of these families. These results will be used to calculate the second Betti numbers of $N$.

The main result of the section is the following

\begin{thm} \label{fundamental group thm} Let $|C|$ be a genus $g \ge 2$ linear system on a general Enriques surface $T$ let $v$  and $N$ be  as in Theorem \ref{thm smooth}.
Then there is a surjection
\[
\Z/(2) \twoheadrightarrow \pi_1(N)
\]
which is an isomorphism in case 
$|C|$ satisfies Assumption \ref{assumption}.
\end{thm}

\begin{rem}
Recall that a hyperelliptic linear system on a general Enriques surface $T$ is always primitive (Prop. \ref{hyperelliptic enriques}). Hence, the theorem above holds unconditionally if $|C|$ is hyperelliptic.
\end{rem}


There are two main ingredients in the proof of this result. The first is a theorem of Leibman \cite{Leibman}, as used also in \cite{Mark_Tik} and in \cite{ASF12}.  We combine this with the second ingredient, which is the Abel--Jacobi map and which allows the comparison of the fundamental group of a family of mildly singular curves with the fundamental group of the corresponding relative compactified Jacobian. We followed an idea of the referee to use the Abel--Jacobi map, as it seemed to be very natural. I am grateful to the referee for this suggestion.
At the end of the section we also prove a result on the vanishing cycles of these families (correcting a mistake that appeared in the first version and that was pointed out to us by the referee). This will be used in Section \ref{second betti number}.

\subsection{Preliminaries}

By abuse of notation, let us denote by
\[
\pi: Y \to |C|,
\]
the map induced by the support morphism $M \to |D|$, and by
\[
\nu: N \to |C|,
\]
the support morphism for the moduli space of sheaves on $T$.
There is a commutative diagram,
\be \label{fibration N and Y}
\xymatrix{
N  \ar[dr]_\nu \ar[rr]^\Phi && Y \ar[dl]^\pi\\
& |C|
}
\ee
which shows that the double cover $\Phi$ restricts fiberwise to a non-trivial double cover.
For later use, let us define a torsion line bundle on $Y$ by setting
\be \label{define L}
\Phi_*{\mc O}_N\cong {\mc O}_Y\oplus L.
\ee
Notice that $L^{\otimes 2}\cong{\mc O}_Y$, and that $L$ generates the kernel of $\Phi^*: \Pic(Y) \to \Pic(N)$.

Let
\[
U \subset |C|, \quad \text{ and  }\,\ \,\,  V \subset |C|,
\]
be the open loci of smooth and reduced curves respectively. For any $t \in U$, let $C_t$ be the smooth member of $|C|$ corresponding to $t$, and set
\[
D_t=f^{-1}(C_t).
\]
By \cite{MumfordPryms}, (vi) Section 2 and Corollary 2 Section 3, the fixed locus $\Fix(\iota)$ of $\iota^*$ acting on  $\Jac(D_t)$ is exactly $f^*(\Jac(C_t))$ and the double cover $N_t \to Y_t$ is induced by the sequence
\be \label{pullback jacobians}
1 \to \Z/(2) \to \Jac(C_t) \stackrel{f^*}{\rightarrow} \Fix(\iota)\subset \Jac(D_t).
\ee

\subsection{Leibman}

Let us start by considering the setting of \cite{Leibman}, which we formulate directly in the context of algebraic varieties. Let $p\colon E \to B$ be a surjective morphism of smooth connected varieties. Assume that $p$ has a section $s$. Let $W \subset B$ be a locally closed smooth subvariety of codimension at least one. Set $U = B \setminus W$, $E_U = p^{-1}(U)$, and $E_W=p^{-1}(W)$. Assume that $E_U \to U$ is a smooth fibration that is topologically locally trivial with path connected fiber $F \stackrel{j}{\hookrightarrow} E_U$. We will say that a morphism $E \to B$ satisfies Leibman's condition if it satisfies the assumptions just mentioned. 
Fix base points $o \in E_U$ and $p(o) \in U$ with respect to which we consider fundamental groups. We have the following commutative diagram
\be \label{leibman diagram}
\xymatrix{
1 \ar[r] & R  \ar[d] \ar[r] & G \ar[r] \ar[d] & H \ar[d]  & \\
1 \ar[r] & \pi_1(F) \ar[r]^{j_*} \ar[d] & \pi_1(E_U) \ar[r]^{p_*} \ar[d] & \pi_1(U) \ar[r] \ar[d] & 1\\
1 \ar[r] & K \ar[r] & \pi_1(E) \ar[r]^{p_*} & \pi_1(B) &\\
}
\ee
where $G=\ker[ \pi_1(E_U) \to \pi_1(E)]$, $H=\ker[\pi_1(U) \to \pi_1(B)]$, $K=\ker[\pi_1(E) \to \pi_1(B)]$, and $R=\ker [\pi_1(F) \to K]$. Since removing closed algebraic subsets only makes the fundamental group larger,  the two vertical arrows on the bottom left are surjective and hence so is $\pi_1(E) \to \pi_1(B)$. 

Following Leibman, let us select a set of generators of $H$ which we will then lift to $G$. A loop in $U$ that can be closed in $B$ can be represented as the image of the boundary of a map $D \to B$ from a two dimensional disk $D$. Choose a general point $x_i$ on every irreducible component $W_i \subset W$ and a small two dimensional disk $D_i \subset B$ transversal to $W_i$ in $x_i$ and such that $D_i \cap W_i=\{x_i\}$.  By transversality, any map from a two dimensional disc to $B$ with boundary contained in $U$ can be moved, up to homotopy, to a map whose image is a disc that is transversal to every component of $W$. Moreover, it can be arranged so that this disc meets every component at the chosen points (cf. \cite[(1.11)]{Leibman}). For every $i$, join the base point $p(o)$ with $\partial D_i$ in every possible way (up to homotopy) via paths $\gamma$ in $U$. The set of paths of the form $\gamma \partial D_i \gamma^{-1}$, together with their inverses, gives a set of generators for $H$. Since $E \to B$ has a section, the set of loops in $E_U$ of the form $s_*(\gamma \partial D_i \gamma^{-1})$ are a lift to $G$ of the generators of $H$. Hence the morphism $G \to H$ is surjective. As a consequence, $\pi_1(F) \to K$ is also surjective and there is an exact sequence
\be \label{R for general case}
1 \to R=\pi_1(F) \cap G \to \pi_1(F) \to \pi_1(E) \to \pi_1(B) \to 1.
\ee
Our aim is the describe the group $R$ more explicitly when $E \to B$ is a family of curves or its relative compactified Jacobian. Before doing so let us point out two important facts.

\begin{rem} \label{remark on birational class}
The first remark is that so far we have only used that the section is defined at the general point of each component of $W$. The second is if $E'$ is a smooth variety and $h: E \dashrightarrow E'$ is a birational map then $\pi_1(E) \cong \pi_1(E')$. In particular, if $h$ restricts to an isomorphism over the general fiber of $E \to B$ we  are free to consider $\pi_1(E')$ instead of $ \pi_1(E)$ in the exact sequence above.
\end{rem}

As usual, let $|C|$ be a linear system on a general Enriques surface $T$ of genus $g \ge 2$. Let $B \subset |C|$ be a general linear subsystem with the property that the universal family
\[
\mc C_B \subset B \times T, \quad p: \mc C_B \to B,  
\]
of curves is smooth and has a section $s$. This is the case, for example,  if $|C|$ is a hyperelliptic linear system (Corollary \ref{section}), or if $B \subset |C|$ is a general pencil in an arbitrary linear system of genus $g \ge 3$ (a general pencil in a genus $g$ linear system has $2g-2$ simple base points). Then $\mc C_B \to B$ satisfies the assumptions of Leibman and we can consider the corresponding sequence (\ref{R for general case}). For the base point $o \in \mc C$, let $C_{t_o}$ be the fiber of $p$ over  $t_o:=p(o)$.  We have the following

\begin{lemma} \label{lemma R for C}
Let $D_{t_o}=f^{-1}(C_{t_o})$ be inverse image of $C_{t_o}$ under the universal cover $f: S \to T$. There is an exact sequence
\be \label{R for C}
1 \to R_{\mc C} \to \pi_1(C_{t_o}) \to \pi_1(\mc C_B) \to 1.
\ee
where
\[
R_{\mc C} =f_*( \pi_1(D_{t_o})),
\]
\end{lemma}
\begin{proof}
This is just (\ref{R for general case}) applied to $\mc C_B \to B$ together with the fact that $\pi_1(B)$ is trivial.
It is easy to see that the second projection $\mc C_B \to T$ induces an isomorphism at the level of fundamental groups so
\[
\pi_1(\mc C_B ) \cong \Z \slash (2).
\]
As a consequence, the two morphisms $\pi_1(C_{t_o}) \to \pi_1(\mc C_B)$ and $\pi_1(C_{t_o}) \to \pi_1(T)$ have the same kernel. Since the $2:1$ cover $S \to T$ restricts to the non--trivial $2:1$ cover $D_{t_o} \to C_{t_o}$, it is clear that $\ker[\pi_1(C_{t_o}) \to \pi_1(T)]=f_*(\pi_1(D_{t_o}))$ and the Lemma follows.
\end{proof}

Now let $H$ be a polarization on $T$ and let
\[
\nu: \ov J_{H,0}(\mc C_B) \to B
\]
be the degree zero relative compactified Jacobian of $\mc C_B \to B$. In other words, $ \ov J_{H,0}(\mc C_B)=N \times_{|C|} B$, where $N$ is the moduli space for the Mukai vector $w=(0, C, -g+1)$. Let us assume that $H$ is  $v$--generic  (so that $N$ and $\ov J$ are smooth) and also such that $\nu$ has a section. Such a polarization exists because of Remark \ref{section defined}.


We are thus in the setting of Leibman. Sequence (\ref{R for general case}) becomes
\be \label{R for J}
0 \to R_{  \ov J_{H,0}(\mc C_B)} \to \pi_1(J_{t_o})=H_1(C_{t_o}, \Z) \to  \pi_1( \ov J_{H,0}(\mc C_B)) \to 0.
\ee

\begin{lemma} \label{independent of H and d}
Let $\mc C_B \to B$ be as above. For any polarization $H$ and any degree $d$ such that $\ov J_{H,d}(\mc C_B)$ is smooth, there is a short exact sequence (\ref{R for J}) with first term $R_{  \ov J_{H,0}(\mc C_B)} $ independent of $H$ and $d$.
\end{lemma}
\begin{proof}
Using Remark \ref{remark on birational class} we only need to check that the birational class of $\ov J_{H,0}(\mc C_B)$ is independent of $H$ and $d$. Independence of $H$ follows from Lemma \ref{existence stable sheaves} and  the independence of $d$ follows from the existence of a section.

\end{proof}

As a consequence to compute $\pi_1(  \ov J_{H,d}(\mc C_B))$ we can drop $H$ from the notation and  only consider the degree $0$ compactified Jacobian. This will be denoted by
\[
\ov J_B:= \ov J_{H,0}(\mc C_B)
\]

Our aim is to use Lemma \ref{lemma R for C} to compute $R_{\ov J_B}=R_{\ov J_{H,0}(\mc C_B)}$ and we will use Abel--Jacobi maps to compare (\ref{R for C}) and (\ref{R for J}).

 Let $U_B=U \cap B \subset B$ be the open  locus parametrizing smooth curves and consider the restriction $J_{U_B}=\nu^{-1}(U_B)$ of $\ov J_B \to B$ to $U_B$.
  Using the section $s : B \to \mc C_B$ we can define an Abel--Jacobi map
\be \label{Abel Jacobi}
\begin{aligned}
A=A_{U_B,s}: \mc C_{U_B} &\longrightarrow J_{U_B},\\
c & \longmapsto m_{C_b, c} \otimes \mc O_{C_b}(s(b)),
\end{aligned}
\ee
which is well known to be an embedding. We can view $A$ as a rational map
\[
\mc C_B \dashrightarrow \ov J_B
\]
which induces, since we are assuming $\ov J$ to be smooth (so the fundamental group is a birational invariant; $A$ is defined on an open set whose codimension is at least $2$) a morphism $A_*: \pi_1(\mc C_B ) \to \pi_1(\ov J)$ which fits into the following commutative diagram
\be \label{Leibman seq for C and J}
\xymatrix{
 1 \ar[r] & R_{\ov J_B} \ar[r]& \pi_1(J_{t_o}) \ar@{->>}[r] &  \pi_1(\ov J_B) \ar[r]& 1\\
 1 \ar[r] & R_{\mc C_B} \ar[r] \ar[u]^{r} & \pi_1(C_{t_o}) \ar@{->>}[u]^{t}\ar@{->>}[r]&  \pi_1(\mc C_B)\ar[u]^{A_*} \ar[r]&1
 }
\ee

\begin{lemma} \label{if r surj then okay}
There exist a surjection $\pi_1(\mc C_B)=\Z \slash (2)  \twoheadrightarrow  \pi_1(\ov J_B)$ which is an isomorphism if and only if $R_{\mc C_B} \to R_{\ov J} $ is surjective. If this is the case, then $R_{\ov J_B}=f_* H_1(D_{t_o}, \Z)$.
\end{lemma}
\begin{proof}
The first two statements are diagram chasing, while the third  follows from the fact that since $R_{\mc C_B}=f_*( \pi_1(D_{t_o}))$ then $r(R_{\mc C_B})=f_* H_1(D_{t_o}, \Z) \subset R_{\ov J_B}$ and this inclusion is an equality if and only if $r$ is surjective.
\end{proof}

To show that the map $\pi_1(\mc C_B) \to \pi_1(\ov J)$ is actually an isomorphism, we will need the following Lemma which we will use after showing that the Abel--Jacobi maps embeds $\mc C_B$ in $\ov J_B$. Before stating the Lemma, let us introduce some more notation to add to the one defined at the beginning of the Section.

 Again, we follow Leibman ($(\gamma')$ on pg 102 and pg 104). For this part we also need that $p: E \to B$ is smooth at the general point of $E_W$ (in the case of $N \to |C|$ this follows from Proposition \ref{MRV on compactified Jac}).

For every component $E^j_{D_i}$ of $p^{-1}(W_i)$ we can choose a general point $q_{ij}$, lifting $x_i$, and a small disk $D_{ij}$, lifting $D_i$,  which is transversal to $E^j_{D_i}$ and only meets it in $q_{ij}$. This is possible because $p$ admits local sections at the general point of every component of $E_W$. For every $i$, we can choose these lifts so that the one corresponding to the component meeting the section $s(B)$ is precisely $s(D_i)$. Moreover, for any path $\eta$ joining $p(o)$ to $\partial D_i$ as above we can choose (since the fibers over $U$ are path connected) a path $\gamma$  in $E_U$ which lifts $\eta$ and which joins the base point in $E_U$ to a fixed point $o_{ij} \in \partial D_{ij}$. Notice that we can chose such points  $o_{ij}$ so that for fixed $i$ they lie over the same point $o_i \in \partial D_i$ and so that the point lying in $s(D_i)$ is precisely $s(o_i)$.

This defines other lifts of the generators of $H$, which are not necessarily contained in the image of the section. The observation, which we will make more precise in Subsection \ref{section vanishing cycles}, is that different lifts of the same generator of $H$ differ by a vanishing cycle of the family. Since for the moment we don't need this, we postpone the discussion on the vanishing cycles in a separate subsection.

\begin{lemma} \label{comparing R}
Let us be given two morphisms $p: E \to B$ and $p':E' \to B'$ satisfying the conditions of Leibman and suppose that there are locally closed embeddings $E' \subset E $ and $B' \subset B$ commuting with $p$ and $p'$.  Let $U' \subset B'$ be the locus where $p'$ is smooth and suppose that $U' \subset U$. Denote by $F'$ the fiber of the topologically locally trivial fibration $E'_{U'}$ and suppose that the inclusion $F' \subset F$ induces a surjection at the level of fundamental group. Suppose furthermore that every component of $E_{W}$ contains a component of $E'_{W'}={p'}^{-1}(W')$, where $W'=B' \setminus U'$ and that both $p$ and $p'$ are smooth at the general point of every component of $E'_{W'}$. Let $R=\ker[\pi_1(F) \to \pi_1(E)]$ and $R'=\ker[\pi_1(F') \to \pi_1(E')]$ be as in ( \ref{leibman diagram}). Then the natural morphism $R' \to R$ is surjective.
\end{lemma}
\begin{proof}
We know that $R=G \cap \pi_1(F)$, where $G=\ker[\pi_1(E_U) \to \pi_1(E) ]$ and similarly $R'=G' \cap \pi_1(F')$, where $G'=\ker[\pi_1(E'_U) \to \pi_1(E')]$. We claim that it is enough to show that the natural morphism $G' \to G$ is surjective. Indeed, consider the short exact sequences $1 \to R \to G \to H \to 1$ and the corresponding primed one $1 \to R' \to G' \to H' \to 1$.  Both are exact on the right because $p: E \to B$ and $p':E' \to B'$ satisfy the conditions of Leibman and because of the discussion after diagram (\ref{leibman diagram}). Morever, the morphism between the two fibrations induce a morphism of complexes between the two short exact sequences. It follows that if $G' \to G$ is surjective, then the cokernel of $R' \to R$ is surjected upon by the kernel of the natural morphism $H' \to H$. Since $\pi_1(F') \to \pi_1(F)$ is surjective, then $\pi_1(U')\to \pi_1(U)$ is injective and hence so is $H' \to H$.
As above, we can write any element $\alpha \in G$ as a product of paths of the form $\gamma \partial D_{ij} \gamma^{-1}$ where the discs $D_{ij}$ are as above and the $\gamma$'s are paths in $E_U$ connecting $o$ to $o_{ij} \in \partial D_{ij}$ . Since $p$ is smooth at the general point of $E_{W_i}^j$, we can choose the $D_{ij}$ to be centered at points $x'_{ij} \in E'_{W'}$. Moreover, if locally we trivialize the embedding $E' \to E$, we can use a homotopy to move the disk $D_{ij}$  so that it is actually contained in $E'$. To show that $G' \to G$ is surjective it is therefore sufficient to show that the paths $\gamma$ joining $o$ to $\partial D_{ij} $ are homotopic to paths $\gamma'$ in $E'_{U'}$. Since $F'$ is path connected we can choose a path $\gamma''$ joining $o$ to $\partial D_{ij}$ and lifting $p(\gamma)$. Then $\gamma {\gamma''}^{-1}$ is a loop in $E_U$ which lies in the kernel of $p_*$, i.e., $\gamma {\gamma''}^{-1}=f$ for some $f \in \pi_1(F)$. Since $\pi_1(F') \to \pi_1(F)$ is surjective by assumption, $f$ is homotopic to a loop in $F'$ and hence $\gamma$ is homotopic to $ \gamma':= \gamma'' f$ which is a path in $E'_{U'}$.
\end{proof}

\subsection{Abel--Jacobi maps and the proof of the Theorem}

The next step will be to show, given a family of curves as above, that the Abel--Jacobi map (\ref{Abel Jacobi}) can be extended to an embedding satisfying the assumptions of this lemma.
Before doing so, we need to recall a few facts about extension of Abel--Jacobi maps to singular curves. This topic has been extensively studied. We refer to \cite{Melo-Rapagnetta-Viviani14} and \cite{Caporaso-Coelho-Esteves-08} and the references therein for a more thorough treatment on the topic. Here we limit ourselves to the most basic facts. We start by considering the extension over the locus $V \subset B$ parametrizing singular but integral curves. Following \cite{Melo-Rapagnetta-Viviani14} and \cite{Caporaso-Coelho-Esteves-08}, over $V$
one can extend (\ref{Abel Jacobi}) by considering on $\mc C_V \times_V \mc C_V$ the sheaf
\[
I_\Delta \otimes p_1^* \mc O_{\mc C_V} (\Sigma), \quad \Sigma:=s(B) \subset \mc C_B,
\]
where $I_\Delta$ is the ideal sheaf of the diagonal in $\mc C_V \times_V \mc C_V$. This sheaf defines a flat family of rank one torsion free sheaves of degree zero, parametrized by the second factor $\mc C_V$. As such, it defines a morphism, extending $A_{U_B,s}$, from $\mc C_V$ into the relative compactified Jacobian. For reference, we highlight the following proposition

\begin{prop} \label{abel jacobi for non hyper}
Let $\mc C_B \to B$ be the  family over a general pencil $B \subset |C|$ in a non--hyperelliptic linear system on a general Enriques surface, let $\ov J_B \to B$ the degree zero relative compactified Jacobian of this family. Choose a section of the family and consider the Abel--Jacobi map (\ref{Abel Jacobi}) with respect to this section. Then $A$ extends to an embedding
\[
\mc C \hookrightarrow \ov J
\]
over $B$.
\end{prop}
\begin{proof}
The fact that the morphism $A$ extends was discussed above and the fact that it is an embedding follows from \cite[Thm 1]{Caporaso-Coelho-Esteves-08}.
\end{proof}

 To extend the Abel--Jacobi map over the locus of reducible curves one needs to be more careful as the sheaves of the form $m_{C_b, c} \otimes \mc O_{C_b}(s(b))$ will in general not be semistable.  Melo, Rapagnetta, and Viviani have shown in \cite[Lem 6.1 and Prop. 6.7]{Melo-Rapagnetta-Viviani14} that on a \emph{fixed} curve one can always find a polarization which guarantees stability and hence that, up to suitably choosing the polarization, the assignment
 \[
 C_b \ni c \mapsto m_{C_b, c} \otimes \mc O_{C_b}(s(b)) \in \ov J_{H,0}(C_b)
 \]
 defines an extension of the Abel--Jacobi morphism (notice, however, that if the curve has separating nodes, then the definition has to be tweaked  (\cite[\S 9-10]{Caporaso-Coelho-Esteves-08} \cite[Prop. 6.7]{Melo-Rapagnetta-Viviani14})). However, the polarization for which the Abel--Jacobi map is defined depends on the given curve and hence the construction does not in general work in families. Luckily, for a hyperelliptic linear system on a general Enriques surface we have the following proposition.

\begin{prop} \label{abel jacobi for hyper}
Let $\mc C \to |C|$ be the universal family of a hyperelliptic linear system of genus $g\ge 2$ on a general Enriques surface. Fix $s: |C| \to \mc C$  one of the two sections and let $B \subset |C|$ be the open subset parametrizing curves that are: irreducible; or are the union of two smooth curves meeting in two points (as is the general curve of the component $\Delta_3$ of the discriminant); or are the union of two smooth curves meeting in one point (as is the general curve of the components $\Delta_1$ and $\Delta_2$ of the discriminant). 
There exists a $(-g+1)$--general polarization $H$ such that $A$ extends to a regular embedding over $B$
\[
\mc C_B \hookrightarrow \ov J_{H,0}(\mc C_B).
\]
\end{prop}
\begin{proof}
By the remarks before the proposition, we only need to check the extension of the morphism on the locus parametrizing reducible curves. By Proposition \ref{discriminant locus}, in codimension one the only curves that appear are of the form of the form $\Gamma_1+\Gamma_2$, where
\begin{enumerate}
\item [(i)] $\Gamma_1 \in |2e_1|$ and $\Gamma_2 \in |(n-2)e_1+e_2|$;
\item [(ii)] $\Gamma_1 = e_1$ and $\Gamma_2 \in |(n-1)e_1+e_2|$;
\item [(iii)] $\Gamma_1 = e'_1$ and $\Gamma'_2 \in |((n-1)e_1+e_2)'|$
\end{enumerate}
(if $g=2$ only the last two cases occur).
Recall that the two sections of a hyperelliptic linear system come from its base points which were described in Lemma \ref{base points}.  Let $\Sigma:= s(|C|)$ be the image of the section.
Up to switching cases $(ii)$ and $(iii)$ we can assume that the following intersections hold: in case $(i)$: $\Sigma \cap \Gamma_1=0$ and $\Sigma \cap \Gamma_2=1$; in case $(ii)$: $\Sigma \cap \Gamma_1=0$ and $\Sigma \cap \Gamma_2=1$; in case $(iii)$: $\Sigma \cap \Gamma_1=1$ and $\Sigma \cap \Gamma_2=0$.
We will use Lemma \ref{stability on two components} to check whether in the three cases $(i), (ii)$, and $(iii)$ the sheaves of the form $m_p \otimes \mc O_\Gamma(\Sigma)$ are stable. Notice that the only thing that matters for stability are the intersection numbers
\[
a=H \cdot e_1, \quad \text{and} \quad b=H \cdot e_2
\]
so we will drop $H$ from the notation and only use $a$ and $b$. Since $g(C)=n+1$ and we are considering the degree zero Jacobian, $\chi=-n$. For curves of type $(i)$, stability for line bundles becomes
\[
-\ff{2na}{na+b}  \le \chi_1 \le - \ff{2na}{na+b} +2.
\]
If $a$ and $b$ are such that $na>b$ then
\be  \label{H general}
1<\ff{2na}{na+b}<2
\ee
(by choosing $H$ to be appropriate combination of $e_1$ and $e_2$ this can certainly be achieved) so $H$ is $\chi$--general for $\Gamma$  and a line bundle is stable if and only if $-1 \le \chi_1 \le 0$, i.e. if and only if
\[
(\chi_1, \chi_2)=(-1, -n+3) \quad \text{ or } \quad
(\chi_1, \chi_2)=(0, -n+2).
\]
This means that the sheaf $m_p \otimes \mc O_\Gamma(\Sigma \res{\Gamma})$ is stable when $p$ is a smooth point on either of the two components and hence $A$ extends over the smooth locus of curves of type $(i)$. When $p \in \Gamma_1 \cap \Gamma_2$,  then the sheaf $m_p \otimes \mc O_\Gamma(\Sigma \res{\Gamma})$ is locally free at only one node. By Lemma \ref{stability on two components} stability is equivalent to $\chi_1=-1$ and hence $m_p \otimes \mc O_\Gamma(\Sigma \res{\Gamma})$ is stable since Euler characteristics of the restriction to the two components are $(\chi_1, \chi_2)=(-1, -n+2)$.

It is more subtle to extend the morphism over the locus of curves with separating nodes \cite[Prop. 6.7]{Melo-Rapagnetta-Viviani14} (as are the curves of cases $(ii)$ and $(iii)$). To fix ideas, let us consider case $(ii)$. Stability requires $(\chi_1, \chi_2)=(0, -n+1)$ which is satisfied for sheaves the form $m_p \otimes \mc O_\Gamma(\Sigma \res{\Gamma})$ if and only if $p$ is point belonging to $\Gamma_2 \setminus \Gamma_1 \cap \Gamma_2$. Indeed, in this case $(\chi_1, \chi_2)=(0,-n+1)$, otherwise $(\chi_1, \chi_2)=(-1,-n+2)$. Following \cite[6.7]{Melo-Rapagnetta-Viviani14} and \cite[\S 9-10]{Caporaso-Coelho-Esteves-08}. For $p \in \Gamma_1$ we can set $A(p)=m_p \otimes \mc O_\Gamma(\Sigma \res{\Gamma}) \otimes \mc O_\Gamma(\Gamma_1)$. Here, $\mc O_\Gamma(\Gamma_i)$ denotes the restriction to $\Gamma$ of the divisor in $\mc C_B \to B$ that  lives over $\Delta_1$ and is swept out by the components  $\Gamma_i$ of the curves parametrized by $\Delta_1$. Since, $\deg \mc O_\Gamma(\Gamma_i)_{|\Gamma_i}=-1$ and $\deg \mc O_\Gamma(\Gamma_i) _{|\Gamma_i}=1$, the sheaf $m_p \otimes \mc O_\Gamma(\Sigma \res{\Gamma}) \otimes \mc O_\Gamma(\Gamma_1)$ is stable.
Similarly, to extend the morphism over the second component of curves of type $(iii)$, we need to twist by $\mc O_\Gamma(\Gamma_2)$. 

For a hyperelliptic linear system this can work in families since there is no monodromy among the irreducible components so over  $\Delta_1$ and $\Delta_2$ we can single out the first and second components of every curve. More precisely, the morphism $\mc C_B \to \ov J_B$ will be determined, using the universal property of the moduli space, by the sheaf 
\[
\mc I_\Delta \otimes p_1^*\mc O_{\mc C}(\Sigma) \otimes \mc O_{\mc C_B \times_B \mc C_B}(D_1) \otimes \mc O_{\mc C_B \times_B \mc C_B}(D_2)
\]
on $\mc C_B \times_B \mc C_B$, viewed as a family of sheaves parametrized by the second factor $\mc C_B$. Here $D_1 \subset \mc C_B \times_B \mc C_B$ is the component of  $(p \times p)^{-1} (\Delta_1 \times \Delta_1)$ parametrizing  pairs of points $(x, y) \in \mc C_B \times_B \mc C_B$ both of which belong to $e_1$ and  $D_2 \subset \mc C_B \times_B \mc C_B$ is the component of $(p \times p)^{-1}(\Delta_2 \times \Delta_2)$ 
parametrizing  pairs of points $(x, y) \in \mc C_B \times_B \mc C_B$  both of which belong to $|((n-1)e_1+e_2)'|$.

Finally, by \cite[Thm 1]{Caporaso-Coelho-Esteves-08}  (see also \cite[Fact 6.10]{Melo-Rapagnetta-Viviani14}) the Abel--Jacobi map is an embedding precisely away from the components of a curve which are smooth rational curves whose intersection with the rest of the curve consists in two separating nodes. 
\end{proof}

From the proof of the above proposition we may deduce the following Corollary.

\begin{cor} \label{components C and J}
Let $\Gamma=\Gamma_1+\Gamma_2$ be a general curve in $\Delta_3$, let $H$ be a general polarization satisfying $na>b$ as above and let $\bar J_H(\Gamma)$ be the relative compactified Jacobian of degree $0$ pure sheaves on $\Gamma$, stable with respect to $H$. Then,  the Abel--Jacobi map defined above embeds
$\Gamma_1$ in the component corresponding to $(-1, -n+3)$ and embeds  $\Gamma_2$ in the component corresponding to $(0, -n+2)$. 
\end{cor}

We can finally show

\begin{thm} \label{pi one for hyperelliptic}
Let $|C|$ be a hyperelliptic linear system of genus $g \ge 2$, and let $N$ be as in Theorem \ref{thm smooth}. Then
\[
\pi_1(N) \cong \Z \slash (2).
\]
and the kernel of the natural morphism $\pi_1(C_{t_o}) \to \pi_1(N)$ is equal to $f_*(\pi_1(D_{t_o}))$. 
\end{thm}
\begin{proof} Recall that $|C|$ is a primitive linear system and hence by Proposition \ref{Yoshioka assumption} it satisfies Assumption \ref{assumption}.
By Lemma \ref{independent of H and d} it is enough to look at the case when the degree is zero and when $H$ is general and satisfes (\ref{H general}). Then by Proposition \ref{abel jacobi for hyper} we can apply Lemma \ref{comparing R} to $\mc C_B \to B$ and $N=\ov J_B \to B$, which gives surjectivity of the map $r$ in diagram (\ref{Leibman seq for C and J}). Hence, by Lemma \ref{if r surj then okay} the morphism $A_*: \pi_1(\mc C )=\Z \slash (2) \to \pi_1(N)$ is an isomorphism and $\ker[\pi_1(C_{t_o}) \to \pi_1(N)]=f_*(\pi_1(D_{t_o}))$.
\end{proof}

Let us now come to the case of a non--hyperelliptic linear system. Since there is no section, the strategy is to first compute the fundamental group of the relative compactified Jacobian of a general pencil (which admits a section) and then to pass from there to the family over the complete linear system.

\begin{lemma} \label{from pencil to whole}
Let $|C|$ be a genus $g \ge 3$ non--hyperelliptic linear system on a general Enriques surface $T$ and let $N$ be as in Theorem \ref{thm smooth}. Let $B \subset |C|$ be a general pencil, let $\mc C_B \to B$ be the universal family of curves and let $\ov J_B=\ov J(\mc C_B)=N \times_B |C| \to B$ be its degree $d$ relative compactified Jacobian. Then
\[
\pi_1(\ov J_B)\cong \Z \slash(2). 
\]
and there is a surjection
\[
\pi_1(\ov J_B) \twoheadrightarrow \pi_1(N).
\]
\end{lemma}
\begin{proof} The proof is the same as for the case of a hyperelliptic linear system.
Since $\mc C_B \to B$ has a section, the birational class of $\ov J_B$ does not depend on the degree (nor on the polarization) and hence we can assume that $d=0$. By Proposition \ref{abel jacobi for non hyper} there is an embedding $A: \mc C_B \to \ov J_B$ which, by Lemmas \ref{comparing R} and \ref{if r surj then okay} induce an isomorphism  $\pi_1 (\mc C_B) \cong \Z \slash (2)$. The second statement is standard. It can, for example, be proved comparing diagrams \ref{leibman diagram} for the two families and using the fact that $\pi_1(U \cap B) \to \pi_1(U)$ is surjective (or also using \cite[Thm 1.1]{Goreski-MacPherson}).
\end{proof}

This shows that there is a $2:1$ cover of $\ov J_B$. Using the norm map we wish to extend this cover to all of $N$.

Let $U' \subset |C|$ be the locus parametrizing integral curves. The norm map will allow us to extend the covering to $N_{U'}=\ov J(\mc C_{U'})$, showing that the surjection $\Z \slash (2)=\pi_1(\ov J_B) \to N_{U'}=\ov J(\mc C_{U'})$ is an isomorphism. To extend this result to all of $N$, we need to assume that $|C|$ satisfies Assumption \ref{assumption}. Indeed, if this holds then $\codim(N\setminus N_{U'}, N ) \ge 2$ and hence $\pi_1(N_{U'})=\pi_1(N)$.

\begin{prop} \label{norm map}
Let 
\[
\xymatrix{
\mc D \ar[dr] \ar[rr]_{2:1}^f && \mc C  \ar[dl] \\
& B &
}
\]
be a family of \'etale double covers between reduced and irreducible curves with locally planar singularities. Let $d$ be an integer and let $\ov J_d({\mc D}) \to B$ and $\ov J_d({\mc C}) \to B$ be the relative degree $d$  compactified Jacobian of the families $\mc D$ and $\mc C$. There is a natural fiberwise \'etale double cover of $\wt {J_d(\mc C)} \to J_d(\mc C)$, which on the Jacobian of a smooth curve $C_b$ is induced by the index two subgroup
\be \label{fiberwise cover}
f_* H_1(D_b, \Z) \subset H_1(C_b, \Z).
\ee
\end{prop}
\begin{proof}
The proof is based on the norm map, see \cite[\S 6.5]{EGAII} and \cite[\S 21.5]{EGAIViv}. Let $J_d({\mc D}) \subset \ov J_d(\mc D)$ and $J_d({\mc C}) \subset \ov J_d(\mc C)$ be the open loci parametrizing locally free sheaves. For simplicity, let us consider the case where the general curve is smooth and where the total space $J_d({\mc C})$ is also smooth (this is the only case we will need; in any event, the general case can be deduced from this using versal deformation spaces). Since the codimension of the complement of $J_d({\mc C})$ in $ \ov J_d(\mc C)$ is of codimension $\ge 2$, $\pi_1(J_d({\mc C}))=\pi_1(\ov J_d({\mc C}))$ and it will be enough to study the double cover over $J_d({\mc C})$ as it will automatically extend to the whole $\ov J_d({\mc C})$. The norm map is the morphism defined by (cf. \cite[\S 6.5]{EGAII})
\be \label{norm map 1}
\begin{aligned}
N_{\mc D \slash \mc C}: J_d({\mc D}) & \longrightarrow J_d({\mc C}) \\
L & \longmapsto \det f_* L \otimes (\det f_* \mc O_{\mc D})^{-1}
\end{aligned}
\ee
For $b \in B$, let $D_b \to C_b$ be the restriction to $b$ of the $2:1$ cover. If  $L=\mc O_{D_b}(\sum n_i x_i)$, then \cite[21.5.5]{EGAIViv}
\be  \label{explicit norm}
N_{{\mc D \slash \mc C}}(L)=\mc O_{C_b}(\sum n_i f(x_i)),
\ee
It is well known \cite{MumfordPryms}, that over the locus of smooth curves the kernel of the norm map has two connected components. To show that this is the case also for a singular (but irreducible) curve we can argue in the following way.  Suppose that $D_b$ and $C_b$ are singular curves and let $n: \widehat C_b \to C_b$ and $m: \widehat D_b \to D_b$ be their normalizations. The norm map is compatible with the pullback to the normalizations (cf. \cite[Prop. 6.5.8]{EGAII}), in the sense that the following is a commutative diagram of short exact sequences of groups
\be \label{norm map 2}
\xymatrix{
 0 \ar[r] &  G \times G \ar[r] \ar[d]^\sum & J_{ D_b} \ar[d]^{N} \ar[r]^{m^*} & J_{\wt D_b} \ar[d]^{\wt N}  \ar[r] &  0 \\
 0 \ar[r] &  G  \ar[r] &J_{ C_b} \ar[r]^{n^*} & J_{\wt C_b} \ar[r] & 0
}
\ee
Here $G\cong n_* \mc O^*_{\wt C_b} \slash \mc O^*_{\wt C_b}$ and the two factors of $G \times G \cong m_* \mc O^*_{\wt D_b} \slash \mc O^*_{\wt D_b}$ are exchanged by the involution. By definition of the norm map, if $z=(x,y) \in G \times G$ then $N(z)$  is the determinant of the endomorphism $\mu_z: G \times G \to G \times G$ determined by the multiplication by $z$. Given that the algebra structure on $G \times G$ is simply the product structure, multiplication by $z=(x, y)$ is just multiplication by $x$ on the first component and by $y$ on the second so $\det \mu_z=xy$. This shows that the restriction of $N$ to $G \times G$ is the product map. From (\ref{norm map 2}) it then follows that $\ker N$ has two connected components, which shows that if we consider the quotient $\wt{J({\mc C})} $ of $J_d(\mc D)$ by the identity component of the norm map we get a commutative diagram
\be \label{two to one norm}
\xymatrix{
J_d({\mc D})  \ar[dr] \ar[rr] && J_d({\mc C})  \\
&  \wt{J_d({\mc C})}  \ar[ur]_{2:1}^\Phi &
}
\ee
which is the desired $2:1$ fiberwise cover and which can be interpreted as the Stein factorization of (\ref{norm map 1}). 
We are left with determining what this double cover is for the Jacobian of a smooth curve $C_b$. Since by construction $ J_d({\mc D})  \to \wt{J_d({\mc C})} $ has connected fibers, it is clear that $\im[ \pi_1(J_{ D_b}) \to \pi_1(J_{ C_b})]=\im[ \pi_1(\wt{J_{\mc C}} \res{b}) \to \pi_1(J_{ C_b})]$. If we consider degree $d$ Abel--Jacobi maps  $D_b \to J_d({ D_b})$ and $C_b \to J_d({ C_b})$ with respect to points $x \in D_b$ and $f(x) \in C_b$, then they will be compatible with $f$ and with the norm map. Using this we can see that on each smooth fiber $\im[ \pi_1(\wt{J_{\mc C} }\res{b}) \to \pi_1(J_{ C_b})]=f_* H_1(D_b, \mathbb Z) \subset H_1(C_b, \mathbb Z) $ as desired.
\end{proof}

\begin{cor} \label{pi one non hyper}
Let $|C|$ be a non--hyperelliptic linear system on a general Enriques surface and let $N$ be as in Theorem \ref{thm smooth}. Then there is a surjection
\[
\Z \slash (2) \to \pi_1(N)
\]
which is an isomorphism in case $|C|$ satisfies Assumption \ref{assumption}.
\end{cor}
\begin{proof} The existence of a surjection $\Z \slash (2) \to \pi_1(N)$ for a non--hyperelliptic linear system holds, unconditionally, thanks for Lemma \ref{from pencil to whole}.
 If a non--hyperelliptic linear system satisfies Assumption \ref{assumption}, then the locus in $N$ parametrizing sheaves supported on non integral curves has codimension $\ge 2$, so we can remove it without affecting the fundamental group. Proposition \ref{norm map} shows that there is a surjection $\pi_1(N) \to \Z \slash (2)$, which means that the morphism $\Z \slash (2)=\pi_1(N_B) \to \pi_1(N)$ has to be an isomorphism.
\end{proof}

Using the norm map we can also give the following geometric interpretation to the universal cover
\be \label{psi}
\Psi: \Nt \to N,
\ee
of $N$ (under the hypothesis that $|C|$ satisfies Assumption \ref{assumption} so that $\pi_1(N)=\Z \slash (2)$).

\begin{prop} \label{universal cover and norm map}
Let $|C|$ be a genus $g \ge 2$ linear system on a general $T$ and suppose that it satisfies Assumption \ref{assumption}. For every $t$ in the open set $U$ parametrizing smooth curves, we let $C_t$ be the corresponding curve and we set $D_t=f^{-1}(C_t)$. Then
\[
\ker[\pi_1(N_t) \to \pi_1(N)]=f_* H_1(D_t, \Z).
\]
In addition, over the locus $U'$ of integral curves the universal cover (\ref{psi}) agrees with the $2:1$ cover $\Phi$ induced by the norm map  (\ref{two to one norm}) .
\end{prop}
\begin{proof}
If $|C|$ is not hyperelliptic, then this is a corollary of the construction of $\Psi$ as an extension of $\Phi$. Suppose therefore that $|C|$ is hyperelliptic. The first statement follows from Proposition \ref{norm map} and from the second statement Theorem \ref{pi one for hyperelliptic}. Let us consider the two fiberwise  coverings  $\Psi$ and $\Phi$ of $N_{U'}$. They are defined by two surjections $\eta_\Phi: \pi_1(N_{U'}) \to \Z \slash (2)$ and $\eta_\Psi: \pi_1(N_{U'}) \to \Z \slash (2)$ which we want to prove are the same morphism. By the first statement we know that $\Psi$ and $\Phi$ induce the same cover over the Jacobians of smooth curves (in fact, it is not hard to show that they induce the same covering also for $1$--nodal curves). Since the surjection $\pi_1(N_U) \to \pi_1(U)$ is split, $\Psi$ and $\Phi$ define the same covering of $N_U$.
Letting $j_*: \pi_1(N_U) \twoheadrightarrow \pi_1(N_{U'}) $ be the natural morphism associated to the open embedding $j: N_U \to N_{U'}$, we may deduce that  $\eta_\Phi \circ j_*= \eta_\Psi \circ j_* $. Since $j_*$ is surjective this shows that $\eta_\Phi=\eta_\Psi$ and hence that $\Psi$ and $\Phi$ define the same covering of $N_{U'}$.
\end{proof}

We finish this section with a results that will be used in Section \ref{second betti number}.

\subsection{Vanishing cycles} \label{section vanishing cycles}
Let us now come to the result on vanishing cycles which will be used in Section \ref{second betti number}. We use the notation introduced at the beginning of this Section and before Lemma \ref{comparing R}.


Recall that for every irreducible components $W_i$ of $W$ we have chosen discs $D_i\subset B$ which are transversal to $W_i$ at a chosen point $x_i$ and we have picked a point $o_i \in \partial D_i$.
 For every $i$,  let us consider the restriction
 \[
 E_{D_i} \to D_i
 \]
 and choose $s(o_i)$ as base point for $\pi_1(E_{D_i})$. Consider the usual specialization map
\be \label{local vanishing cycles}
{sp_i}_*: \pi_1(F) \to \pi_1(E_{D_i}) \cong \pi_1(E_{x_i}), \quad E_{x_i}=p^{-1}(x_i),
\ee
where the isomorphism $\pi_1(E_{D_i}) \cong \pi_1(E_{x_i})$ comes from a retraction $E_{D_i} \to E_{x_i}$.
Let $V_i$ be a set of generators of  $ \ker {sp_i}_*$ (which by definition is the group the vanishing cycles of the family $E_{D_i}$).  Given a loop $v_i \in V_i$ based in $s(o_i)$  and any path $\gamma$ joining $p(o)$ to $ o_i \in \partial D_i$ as above, we can form a loop  $v $ in $ E_U$ by setting
\[
v=s_*(\gamma) v_i s_*(\gamma)^{-1}.
\]
Denote by $V$ the set of paths obtained by doing this for every component of $W$, and by $V^\rho$ the normal subgroup generated by $V$ and all its conjugates under the monodromy action of $\pi_1(U)$ on $\pi_1(F)$.

\begin{prop} \cite{Leibman} \label{what is R}
$R=[\pi_1(F), s_*H]  \cdot V^\rho$ (here the commutator $[\pi_1(F), s_*H]$ is taken in $\pi_1(E_U)$ and, since $\pi_1(F) \subset \pi_1(E_U)$ is normal, it is contained in $\pi_1(F)$).
\end{prop}
\begin{proof} This proof imitates the proofs of \cite[Lem. 1.2 and 1.7]{Leibman}.
An element $\alpha$ in $R=\pi_1(F) \cap G$ can be represented as the boundary of a map $\varphi: D \to E$, with $o \in \varphi (\partial D) \subset E_U$ and such that  $p_*(\varphi (\partial D))=1$ in $\pi_1(U)$. By transversality, $\varphi(D)$ can be made transversal to each component $E^j_{W_i}$ at the chosen points $q_{ij}$ and such that the intersection of $\varphi(D)$ with a suitable neighborhood of $q_{ij}$ is contained in $D_{ij}$ for every $i$ and $j$. Choose an orientation for $\partial D_{ij}$. We can write $\alpha=\prod \alpha_{ij}$ where
\[
\alpha_{ij}=\beta_{ij} \partial D_{ij} {\beta_{ij}}^{-1}
\]
for a path $\beta_{ij}$ in $E_U$ joining $o$ to the chosen point $o_{ij} \in \partial D_{ij} $. Notice that to write $\alpha$ as a product we have defined an ordering on the set of bi-indices $ij$, i.e. we have chosen a bijection
$\nu:\{ ij\} \to L=\{1, \dots, N\}$  such that we can write
\[
\alpha=\prod_{{(ij) \, | \, \nu(ij) =1}}^{N} \alpha_{ij} .
\]

Join $s(o_i) \in s_*(D_i)$ to $o_{ij}$ via a path $\epsilon_{ij}$ in the fiber $E_{x_i}$. Set $\gamma_{ij}:=p_* \beta_{ij}$, so that $\gamma_{ij}$ is a path in $U$ joining $p(o)$ to $p(o_i)$. Then $\beta_{ij} {\epsilon_{ij}}^{-1}(s_*\gamma_{ij})^{-1}=f_{ij}$, for some $f_{ij} \in \pi_1(F)$. Since $p_*(\epsilon_{ij} \partial D_{ij} {\epsilon _{ij}}^{-1})=p_*(s_* \partial D_i)$, the difference
\[
v_{ij}:=s_*( \partial D_i)^{-1} \epsilon_{ij} \partial D_{ij} {\epsilon _{ij}}^{-1}
\]
lies in $\ker({sp_i}_*)$. This shows we can write
\[
\alpha_{ij} =\beta_{ij} \partial D_{ij} {\beta_{ij}}^{-1}= \beta_{ij}  {\epsilon _{ij}}^{-1} \underbrace{ {\epsilon _{ij}} \partial D_{ij}  {\epsilon _{ij}}^{-1}}_{s_*( \partial D_i )v_{ij}  } {\epsilon _{ij}} {\beta_{ij}}^{-1}=
\]

\[
=\beta_{ij}  {\epsilon _{ij}}^{-1}  s_*( \partial D_i )v_{ij}   {\epsilon _{ij}}{\beta_{ij}}^{-1}=
f_{ij} s_*(h_{ij}) {f_{ij}}^{-1} w_{ij},
\]
where $h_{ij}= \gamma_{ij} \partial D_i \gamma_{ij}^{-1}$ and
\[
w_{ij}={\beta_{ij}} {\epsilon _{ij}}^{-1}  v_{ij}  {\epsilon _{ij}} {\beta_{ij}}^{-1}.
\]
We can write
\[
w_{ij}=s_*(\gamma_{ij}) d_{ij} v_{ij} {d_{ij}}^{-1}s_*(\gamma_{ij})^{-1},
\]
where $d_{ij}:=s_*(\gamma_{ij})^{-1} \beta_{ij} {\epsilon _{ij}}^{-1} $ is homotopic to a loop in the fiber $p^{-1}(o_i)$. Since clearly $d_{ij} v_{ij} {d_{ij}}^{-1} \in \ker({sp_i}_*)$,  we have $w_{ij} \in V$. We can thus write 
\[
\alpha=\prod_{ij} [f_{ij}, s_*(h_{ij}) ] s_*(h_{ij}) w_{ij}=\prod_{\nu=1}^N [f_{\nu}, s_*(h_{\nu}) ] s_*(h_{\nu}) w_{\nu},
\]
where in the last equality we have used the ordering of the bi--indices introduced above. 
Following Leibman (pgg. 100 and 103) we can write 
\[
\alpha=\prod_{\nu=1}^N [{f'}_{\nu}, s_*(h'_\nu) ]  {w'}_{\nu} \prod_{\nu=1}^N s_*(h_\nu),
\]
where
\[
f'_\nu:=\Big(\prod_{\xi=1}^{\nu-1} s_*(h_\xi) \Big) f_\nu \Big(\prod_{\xi=1}^{\nu-1} s_*(h_\xi) \Big)^{-1}
\]
\[
s_*h'_\nu:= \Big(\prod_{\xi=1}^{\nu-1} s_*(h_\xi) \Big) s_*h_\nu  \Big(\prod_{\xi=1}^{\nu-1} s_*(h_\xi) \Big)^{-1}
%
%
\]
\[
w'_\nu:= \Big(\prod_{\xi=1}^{\nu} s_*(h_\xi) \Big) w_\nu  \Big(\prod_{\xi=1}^{\nu} s_*(h_\xi) \Big)^{-1}
\]
Notice also that $p_* \prod_{\nu=1}^N s_*(h_\nu)=1$ and since $p_* s_*=\id$, then $\prod_{\nu=1}^N s_*(h_\nu)=1$, and hence
\[
\alpha=\prod_{\nu=1}^N [{f'}_{\nu}, s_*(h'_\nu) ]  {w'}_{\nu}.
\]
To finish the proof we only have to notice that ${w'}_{\nu} \in V^\rho$. This follows from that fact that if $h \in \pi_1(U)$ and $w \in \pi_1(F)$, then $s_*(h) w s_*(h)^{-1}$ is exactly the monodromy action $\rho_h(w)$ of $h$ on $w$ (cf. \cite{ASF12}, \S 7, in particular, (7.7) and (7.9)). Since we have already observed that $w_\nu \in V$, it follows that ${w}_{\nu} \in V^\rho$.


\end{proof}

Let us now apply this result when $[E \to B]=[\mathcal C\to B]$ is a smooth family of reduced  curves whose general member is smooth  of genus $g$ and such that in codimension one we have only nodal curves of geometric genus $g-1$ (this is indeed the case for the  families we consider).  Suppose the family has a section so that we can apply the remarks above. Consider, as we did before Lemma \ref{comparing R}, generators of $H$ of the form $h_{i, \gamma}= \gamma \partial D_i \gamma^{-1}$. For each $h_i$, the monodromy operator $\rho_{h_i}$ on the fiber $ C_{t_o}$, for $t_o:=p(o)$, is the Dehn twist around a closed loop $c_i=c_{i, \gamma} \subset  C_{t_o}$ , called the vanishing cycle associated to $h_i$ (see \cite[Chapter XI]{GACII}). The choice of $c_i$ depends on the choice of local coordinates, but its  homology class  is well defined (up to a sign). To see how it compares to the local vanishing cycles of $\ker({sp_i}_*)$ (cf. (\ref{local vanishing cycles})), consider a $v_i  \in V_i$ and  the loop $ s_*( \gamma) v_i  s_*(\gamma)^{-1}$, which clearly belongs to $R$. Since we are assuming that in codimension one the genus drops only by one, the image in homology of $\ker({sp_i}_*)$ has rank one and hence the classes of $s_*( \gamma) v_i  s_*(\gamma)^{-1}$ in $H_1(C_{t_o}, \Z)$ are all a multiple of $c_i$. In particular, the classes of the ${w'}_{ij}$ defined in the previous Proposition will also be a multiple of $c_i$.

%
Also, for any $f \in \pi_1(C_{t_o})$ we have, passing to $H_1(C_{t_o}, \Z)$, switching to the additive notation, and using that $[f, s_*(h_i)]=f \rho_{h_i}(f^{-1}).$
\be \label{commutator and monodromy}
[f, s_*(h_i) ] =f-f-(f, c_i)c_i=-(f, c_i)c_i \in H_1(C_{t_o}, \Z).
\ee

\begin{cor} \label{generated by vanishing cycles}
Let $p: \mc C \to B$ be a family of curves as above.  And let $R_{\mc C}=\ker[\pi_1(C_{t_o}) \to \pi_1(\mc C)]$ be as in (\ref{leibman diagram}).
Then the image of $R_{\mc C}$ in $H_1(C_{t_o}, \Z)$ is generated by the vanishing cycles associated with a set of generators of $H$.
\end{cor}
\begin{proof}
This is an immediate consequence of Proposition \ref{what is R}, of the definition of vanishing cycle, and of (\ref{commutator and monodromy}).
\end{proof}

\section{The canonical bundle} \label{canonical bundle}

The aim of this section is to show that the canonical bundle of $N$ is trivial. We start with the following adaptation of Theorem 8.3.3 of \cite{Huybrechts-Lehn} to our context

\begin{prop} [\cite{Huybrechts-Lehn}]
Let $T$ be an Enriques surface, and let $M$ be a component of a moduli space parametrizing stable sheaves $F$ such that $F \ncong F \otimes \omega_Y$. 
Then the canonical bundle is torsion, i.e.,
\[
\omega_M=0 \,\, \text{ in } \Pic(M)_\Q.
\]
\end{prop}
\begin{proof}
First, notice that $M$ is smooth since by assumption the obstructions vanish. Even though Chapter 8 of \cite{Huybrechts-Lehn} is formulated for sheaves of positive rank, one can go through all the results needed for the proof of Theorem 8.3.3 and check that they work, with the appropriate modifications, also in the case of pure dimension one sheaves. 
\end{proof}

Recall from (\ref{N}),  (\ref{Y}) and (\ref{psi}) the definitions of $N$, $Y$ and $\Nt$.

\begin{cor} With the assumptions of Theorem \ref{thm smooth}, we have
$\omega_{\Nt} \cong \mc O_{\Nt}$
\end{cor}
\begin{proof}
By the proposition above, $\omega_{\Nt}$ is a torsion class in $\Pic(\wt N)$, but since $\wt N$ is simply connected, this class has to be trivial.
\end{proof}

\begin{prop} \label{comparing the canonical} Let the assumptions be as in Theorem \ref{thm smooth}. The following are equivalent\\
$i)$The canonical bundle of $Y$ is trivial;\\
$ii)$The canonical bundle of $N$ is trivial;\\
$iii)$The canonical bundle of $\Nt$ is trivial.
\end{prop}
\begin{proof}
Since $\Phi: N \to Y$ and $\Psi: \Nt \to N$ are \'etale, we only need to prove that $ii)$ implies $i)$ and that $iii)$ implies $ii)$. We start with the first implication, so let us suppose that $\omega_N \cong \mc O_N$.
Recall from formula (\ref{define L}) the definition of $L$.  Then $\Phi^*\omega_Y\cong \mc O_N$, implying that either $\omega_Y$ is trivial, or that it is isomorphic to $L$.  Consider a point $t \in U$ and denote, as usual, by $Y_t$ the fiber over $t$. To conclude we make the following two claims. The first  is that $L_{Y_t}$ is not trivial, which follows immediately from (\ref{pullback jacobians}). 
Whilst the second claim is that $({\omega_Y})_{|Y_t}$ is trivial which follows immediately from the fact that $\omega_{Y_t}$ and $N_{Y_t|Y}$ are trivial. Hence, $\omega_Y \ncong L$.
The same argument applies to show that the canonical bundle of $\Nt$ is trivial if and only if the canonical bundle of $N$ is trivial. In fact, the only thing we need is that on each fiber the double cover $\Nt_t \to N_t$ is non-trivial, and this is a consequence of Corollary \ref{universal cover and norm map}.
\end{proof}

The steps above prove the following theorem

\begin{thm} \label{thm canonical}
Let $|C|$ be a genus $g \ge 2$ linear system on a general Enriques surface $T$ and let $N$ be as in Theorem \ref{thm smooth}. Then
\[
\omega_N\cong \mc O_N.
\]
\end{thm}

Notice that Propositions \ref{comparing the canonical} and \ref{comparing the canonical}, and hence Theorem \ref{thm canonical}, are \emph{not} conditional to Assumption \ref{assumption}.

\begin{cor}\emph{
With the same assumptions as in the theorem above,
\[
\chi({\mc O}_Y)=\chi({\mc O}_N)=0.
\]}
\end{cor}
\begin{proof}
By Serre duality, this is true for any odd--dimensional Calabi--Yau manifold.
\end{proof}

Moduli spaces of sheaves on a K3 surface share many properties of the surface itself. As we saw in Theorem \ref{thm canonical}, this is not the case for Enriques surfaces. Another instance of this lack of analogy is the fact that the universal cover of $N$ induces a \emph{non-trivial} cover of every fiber: any Enriques surface $T$ admits an elliptic fibration $T \to \P^1$ with exactly two multiple (double) fibers. The canonical bundle of $T$ is the difference of the two half fibers and, moreover, the universal cover $f: S \to T$ induces a \emph{trivial} cover of every reduced fiber; indeed, if $e$ is a primitive elliptic curve in $T$, then for any reduced curve $\Gamma $ belonging to $|2e|$, $f^{-1}(\Gamma)$ is the disjoint union of two members of $|f^*e|$. In fact, the covering $S \to T$ is induced by base change via a degree two morphism $\P^1 \to \P^1$ (ramified at the two points corresponding to the non-reduced fibers) whereas in the case of $\nu: N \to |C|$, as we have already mentioned, the restriction of the universal cover to the fibers of $\nu$ is non-trivial.

This difference in behavior appears also in comparison to other types of moduli spaces of sheaves. In \cite{Oguiso_Schroer11}, Oguiso and Schr\"oer prove that the Hilbert scheme of $n$ points on a given Enriques surface $T$ has the property that the canonical bundle is not trivial, but twice the canonical bundle is trivial. It would be interesting to know if one could extract a general principle from this phenomenon, i.e., that the canonical bundle of a moduli space depends on the parity of its dimension. It would also be interesting to study, given a genus $g$ linear system $|C|$, the geometry of the rational Abel-Jacobi  maps
\[
T^{[g-1]} \dasharrow N^{g-1}, \quad \,\, \text{ and } \,\,\,\,\,\,\, N^g \dasharrow T^{[g]}.
\]
Here $N^d$ denotes the degree $d$ relative compactified Jacobian of $|C|$.

We end the section with the main result

\begin{thm} \label{calabi-yau}
Let $N$ be as in Theorem \ref{thm smooth}. Then the Calabi-Yau manifold $N$ is irreducible. By this we mean that
 \[
 H^p(N, {\mc O}_N)\cong \left\{\begin{aligned}
 &\C \,\,\,\, \text{if} \,\,\,p=0, \text{ or } p=2g-1,\\
 & 0 \,\,\,\,\, \text{otherwise}.
\end{aligned}
\right.
 \] 
\end{thm}

The main step in proving the theorem is the following proposition that computes the higher direct images of the structure sheaf by using the corresponding result of Matsushita \cite{Matsushita05} for the morphism $\pi: M \to |D|$.

\begin{prop} \label{higher direct images} Let $N$ be as in Theorem \ref{thm smooth}. Then,
\be
R^i\nu_*{\mc O}_N\cong \wedge ^i \oplus_{j=1}^g {\mc O}_{\P^{g-1}}(-1).
\ee
\end{prop}

Since
\[
H^k(\P^{g-1}, {\mc O}_{\P^{g-1}}(-p+k))\cong \left\{\begin{aligned}
&\C \quad \quad  \text{if } (k,p)=(0,0),  \quad \,\, \text{or } \,\,\, (k,p)=(g-1, 2g-1)\\
&0 \quad \quad \,\,\, \text{otherwise}
\end{aligned}
 \right.
\]
the spectral sequence calculating $H^p(N, {\mc O}_N)$ degenerates, and the theorem easily follows.

\begin{proof}[Proof of Proposition \ref{higher direct images}]
Since the canonical bundles of $N$ and of $Y$ are trivial, by Theorem 2.1 in \cite{Kollar1} the sheaves $R^i\nu_*{\mc O}_N$  and $R^i\pi_*{\mc O}_Y$ are torsion free and by Corollary 3.9 in \cite{Kollar86} they are reflexive. 
Moreover, $\Phi$ is a finite  morphism so the spectral sequence associated to the composition of functors yields an isomorphism (recall that $L$ was defined in (\ref{define L}))
\[
R^i\nu_*{\mc O}_N \cong  R^i\pi_*{\mc O}_Y \oplus R^i\pi_*L.
\]
However, since $L_{|Y_t}$ is a non-trivial torsion line bundle, the higher direct images $R^i\pi_*L$ are supported on the discriminant locus of $|C|$ and since the sheaf $R^i\nu_*{\mc O}_N $ is torsion free, they have to be identically zero. It follows that
\be \label{N and Y}
R^i\nu_*{\mc O}_N \cong  R^i\pi_*{\mc O}_Y.
\ee
The proof of the proposition can then be deduced from the following three claims, whose proof uses Proposition \ref{degenerations} below.

\subsection*{Claim I} $R^1\nu_*{\mc O}_N \cong R^1p_*{\mc O}_{\mc C}$. 

\subsection*{Claim II} $R^1\nu_*{\mc O}_N\cong \oplus_{i=1}^g {\mc O}_{\P^{g-1}}(-1)$.

\subsection*{Claim III} $R^i\nu_*{\mc O}_N \cong \wedge ^i R^1\nu_*{\mc O}_N$.

For Claim I, first notice that there is a natural  isomorphism over the locus $U \subset |C|$ of smooth curves (e.g. see Lemma \ref{higher direct images different degree} below).
By Proposition \ref{degenerations} below, the isomorphism extends naturally over the general point of every component of the discriminant. Hence, there is an open subset $W \subset |C|$, whose complement has codimension greater or equal to two, over which the sheaves in question are isomorphic. Since they are reflexive sheaves, this isomorphism extends to an isomorphism over all of $|C|$. 

To show Claim III, we first use Proposition \ref{degenerations} as well as Claim I to find an isomorphism which is defined over an open set $W$ as above. Since by Claim II $R^1\pi_*{\mc O}_N$, and hence also its exterior powers, are locally free, the isomorphism defined over $W$ extends to the whole $|C|$ and we have proved the claim.

For Claim II, we argue as follows. Let $\mc I$ denote the ideal sheaf of $|C|$ in $|D|$. Recall that $ \mc I/ \mc I^2\cong \oplus_{i=1}^g {\mc O}_{\P^{g-1}}(-1)$ and that the short exact sequence (on which the involution $\iota^*$ acts)
\[
0 \to \mc I/ \mc I^2 \to (\Omega^1_{|D|})\res{|C|} \to \Omega^1_{|C|} \to 0
\]
is split. The sheaf $ \mc I/ \mc I^2$ is the $\iota^*$--anti--invariant part of  $(\Omega^1_{|D|})\res{|C|}$ and the sheaf $ \Omega^1_{|C|}$ is the $\iota^*$--invariant part. By Theorem 1.3 of \cite{Matsushita05}, there is an isomorphism
\be \label{matsushita iso}
\Omega^1_{|D|} \cong R^1\pi_* {\mc O}_M.
\ee
Since this isomorphism is induced by the symplectic form $\sigma$ of $M$, it interchanges the invariant and anti-invariant subbundles of $\Omega^1_{|D|} \res{|C|}$ and of $R^1\pi_* {\mc O}_M \res{|C|}$. In particular,  the composition of the inclusion $\mc I/ \mc I^2 \to (\Omega^1_{|D|})\res{|C|}  \cong R^1\pi_* {\mc O}_M\res{|C|}$ with the natural morphism $R^1\pi_* {\mc O}_M \res{|C|} \to R^1\pi_* {\mc O}_Y$ is non-zero and generically surjective. We need to show that it is an isomorphism. Consider a general line $\ell \subset |C|$ and let $p': \mc C' \to \ell$ be the restriction of the family of curves to $\ell$, so that $\mc C'$ is just the blow up of $T$ at the $2g-2$ base points of the pencil. By base change (over $W$ all families in question are  flat, so we can apply base change) and by Claim I there is an isomorphism
\[
R^1\nu_* {\mc O_N} \res{\ell}=R^1\pi_* {\mc O_Y}\res{\ell}\cong R^1p'_* {\mc O_{\mc C'}}.
\]
Since $R^1p'_* {\mc O_{\mc C'}}$ is locally free of rank $g$, we can write
\[
R^1p'_* {\mc O_{\mc C'}}=\oplus_{i=1}^g \mc O_{\P^1}(a_i), \quad \text{ for some } a_i \in \Z.
\]
Since the base is one-dimensional, the Leray spectral sequence degenerates and we can use the Hodge numbers of $\mc C'$ to calculate the $a_i$'s. From
\be
\begin{aligned}
&0=H^1(\mc C', {\mc O}) \cong H^1(\P^1, {\mc O})\oplus H^0(\P^1, R^1 {p}_*{\mc O}_{\mc C'}), \,\,\,\,\,\text{and}\\
&0=H^2(\mc C', {\mc O})\cong H^2(\P^1, {\mc O})\oplus H^1(\P^1, R^1 {p}_*{\mc O}_{\mc C'})\oplus H^0(\P^1, R^2 {p}_*{\mc O}_{\mc C'}),
\end{aligned}
\ee
we deduce that
\[
a_i < 0, \,\,\,\,\text{and }\,\,\,\,\,\,a_i>-2.
\]
We conclude that $a_i=-1$ for every $i$, so that $R^1p'_* {\mc O_{\mc C'}}=\oplus_{i=1}^g \mc O_{\P^1}(-1)$. It follows that the morphism
\[
\mc I/ \mc I^2 \to R^1\pi_* {\mc O}_Y
\]
defined above is an isomorphism over an open subset whose complement has codimension greater or equal to two. Hence it extends to a global isomorphism and Claim II follows using (\ref{N and Y}).

\end{proof}

\begin{rem} Analogously to the cases of ${\mc O}_N$ and ${\mc O}_Y$, one can show that also the higher direct images of ${\mc O}_N$ and of ${\mc O}_{\Nt}$ are isomorphic, so that
 \[
 H^p(\Nt, {\mc O}_{\Nt})\cong \left\{\begin{aligned}
 &\C \,\,\,\, \text{if} \,\,\,p=0, 2g-1,\\
 & 0 \,\,\,\,\, \text{otherwise},
\end{aligned}
\right.
 \] 
 and $\Nt$ is an irreducible Calabi-Yau manifold.
\end{rem}

\begin{prop} \label{degenerations}
Let $ \mc C \to B$ be a projective family of smooth genus $g$ curves parametrized by a smooth projective curve (or a disc), and let $p:  \ov{ \mc C} \to \ov{ B}$ be a smooth compactification of the family such that for every point $a_i \in \ov{ B} \setminus B$ the curve $\ov {\mc C}_{a_i}=p^{-1}(a_i)$ is reduced and nodal. Let
$q: \ov{\mc J} \to \ov{B}$ be a relative compactified Jacobian of the family $ \ov{ \mc C}$ and suppose that it is smooth.
There is a natural isomorphism
\be 
R^i q_* \mc O_{\ov{\mc J}} \cong \wedge^i R^1 p_* \mc O_{\ov{\mc C}},
\ee
\end{prop}

The proposition will use the following lemma and some results about Hodge bundles and their degenerations which we  recall in the next subsection.

\begin{lemma} \label{higher direct images different degree}
Let $q: \mc C \to B$ be a family of smooth curves and let and $\nu: J^d_{\mc C} \to B$ be the degree $d$ relative compactified Jacobian. For every $i$ there is a natural morphism
\[
R^i \nu_* \Q_{J^d_{\mc C}} \to R^i q_* \Q_{\mc C},
\]
which is an isomorphism for $i=1$. Moreover, the same holds for the higher direct images of the structure sheaves.
\end{lemma}
\begin{proof}
Suppose $\mc C \to B$ has a section $s$. Then as in (\ref{Abel Jacobi}), we can consider an Abel--Jacobi map $A_s: \mc C \to J^d_{\mc C}$ whose pull--back $A_s^*$ induces a morphism between the local systems. Though the map itself depends on the section $s$, the morphism $A_s^*$ does not, since translation by a point on an abelian variety induces the identity in cohomology. It follows that even if $q$ does not have a section we can choose local sections to define local morphisms which, since they are independent of the section, can be glued to define a global morphism. The same argument can be applied to the direct images of the structure sheaf.
\end{proof}

In Section \ref{second betti number}, there will be a more refined version of this Lemma (Proposition \ref{J and J' MHS}).

\subsection{Degeneration of Hodge bundles} We follow \cite{Zucker}, \cite{Katz71}, \cite{Kollar86},  \cite{Steenbrink77}, and \cite{Peters-Steenbrink} for which we refer for more details and complete proofs.

Let $B$ be a smooth curve and let $f:  Z \to B$ be a smooth projective morphism. The degree $i$th cohomology of this family determines a degree $i$ variation of Hodge structures (VHS for short)  on $B$ whose underlying local system is $R^i f_* \C$. By \cite[I.2.28]{Deligne70} the locally free sheaf 
\begin{equation}
\h^i:=R^i f_* \C \otimes {\mc O}_B
\ee
is isomorphic to the hypercohomology sheaf
\begin{equation}
R^if_* \Omega^\bullet_{Z|B},
\ee
where $\Omega^\bullet_{Z|B}$ is the complex of relative differentials of the family. We denote by 
\begin{equation} \label{Gauss-Maninn}
\nabla: \h^i \to  \h^i\otimes\Omega^{1}_{B}
\ee
the Gauss-Manin connection associated to $R^i f_* \C$. Consider the so called  \emph{filtration b\^ete}
\begin{equation} \label{filtration bete1}
\F^p \Omega^{\bullet}_{Z|B}:=\Omega^{\bullet\ge p}_{Z|B},
\ee
of the complex of relative differentials.
The spectral sequence in hypercohomology associated to this filtration has $E_1^{p,q}$ term equal to
\begin{equation} \label{E2}
R^qf_* \Omega^p_{Z|B},
\ee
and abuts to (the associated graded pieces of)
\begin{equation} \label{hi}
R^{i}f_* \Omega^\bullet_{Z|B}= \h^{i}.
\ee
Since $f$ is smooth these sheaves are locally free and, just as in the case of a smooth projective variety (see for example  \cite[Prop. 10.29]{Peters-Steenbrink}), one can show that the spectral sequence degenerates at $E_1$ and that the maps
\begin{equation}
R^if_* \Omega^{\bullet\ge p}_{Z|B} \to R^if_* \Omega^{\bullet}_{Z|B} 
\ee
are injective. In other words, the filtration induced on $\h^{i}$ is
\begin{equation}\label{filtration}
\F^{p}\h^i=R^if_* \Omega^{\bullet\ge p}_{Z|B}\subset \h^i,
\ee
and its associated graded pieces are the sheaves
\[
R^qf_* \Omega^p_{Z|B}, \quad \text{with}\quad p+q=i.
\]
The filtration (\ref{filtration}) is precisely the Hodge filtration of the variation of Hodge structures of the family $Z \to B$.
Now consider a smooth projective compactification 
\[
f\colon \Zb \to \Bb,
\]
and suppose that $D:=f^{-1}(\Bb \setminus B)$ is a reduced divisor with normal crossing. Let $\Omega_{\Zb|\Bb}^\bullet(\log D)$ be the complex of relative logarithmic differentials (\cite[(21)]{Zucker}).

By a classical theorem \cite[II.7.9]{Deligne70}, it is known that $(\h^i, \nabla)$ is an algebraic differential equation with regular singular points. By definition \cite[(2.1) (i)]{Kollar86}, this means that $\nabla$ has logarithmic poles at every point $b$ in $\Bb \setminus B$,  that is to say, there exists a vector bundle extension $\hb^i$ of $\h^i$ to all of $\Bb$, such that the connection $\nabla$ extends to a morphism
\[
\overline \nabla: \hb^i \to \hb^i \otimes \Omega_{\Bb}^1(\log A ), \quad A:=\ov B \setminus B.
\]
Such an extension is not unique, but as we will see there is a unique one satisfying some additional conditions. Recall \cite[(2.1) (iii)]{Kollar86} that the \emph{residue} $\Res_b (\overline \nabla)$ of $\overline \nabla$ at a point $b$ in $\Bb \setminus B$ is defined to be the endomorphism of the fiber $\hb^i_b$ induced by restricting $\overline \nabla$ to $b$ and composing this restriction with $\id_{\hb^i_b}\otimes \Res$ where $\Res$ is the Poincer\'e residue map $ \Omega_{\Bb}^1(\log b) \to \mathbb C_b $
\[
\Res_b (\overline \nabla): \hb^i_b \to   \hb^i_b \otimes \Omega_{\Bb}^1(\log b) \to \hb^i_b \otimes \C_b.
\] 
In other words, if we fix a local trivialization of $\hb^i$ around $b$ and if $z$ is a local coordinate on $\Bb$, the residue is defined by the equation $\overline \nabla=dz\otimes (\Res_b (\overline \nabla) 1/z + \dots)$. Another important property of the residue \cite[Lem. 2.2]{Kollar86} is that if $T$ denotes the local monodromy operator, then
\be \label{monodromy and residue}
T=\exp(-2\pi i \Res_b(\overline \nabla)).
\ee

By \cite[(2.11)]{Steenbrink77} and \cite{Steenbrink_Zucker80} ((5.1) for the notation and (5.3) for the result; for a clear exposition see \cite{Zucker}), the sheaves 
$$R^if_* \Omega^\bullet_{\Zb|\Bb}(\log D)$$
and
$$R^qf_* \Omega^p_{\Zb|\Bb}(\log D), \quad \,\,\,p+q=i,$$
are \emph{locally free} extensions of (\ref{hi}) and (\ref{E2}), respectively. Moreover, Katz proved in \cite[V]{Katz71} (see also \cite[Thm. 10.28]{Peters-Steenbrink}) that there is a natural morphism
\be \label{extending nabla}
\overline \nabla: R^if_* \Omega^\bullet_{Z|B}(\log D) \to R^if_* \Omega^\bullet_{Z|B}(\log D) \otimes\Omega_{\Bb}^1(\log \Bb \setminus B),
\ee
which extends the Gauss-Manin connection.
It follows that we can set
\[
\hb^i:= R^if_* \Omega^\bullet_{Z|B}(\log D).
\]
On the other hand, the filtration bete (\ref{filtration bete1}) extends to a filtration
\begin{equation} \label{filtration bete2}
\F^p \Omega^{\bullet}_{\Zb|\Bb}(\log D):=\Omega^{\bullet\ge p}_{\Zb|\Bb}(\log D),
\ee
of $R^if_* \Omega^\bullet_{Z|B}(\log D)$. This defines a spectral sequence whose $E_1^{p,q}$ terms are 
\[
R^q f_* \Omega^p_{\Zb|\Bb}(\log D).
\]
Since these sheaves are locally free and the differential is generically zero, the differential is identically zero. Hence the spectral sequence, which abuts to the associated graded pieces of $R^if_* \Omega^\bullet_{Z|B}(\log D)$, degenerates at $E_1$. Moreover, these graded pieces $R^qf_* \Omega^p_{\Zb|\Bb}(\log D)$ are locally free and hence the extensions
\[
\F^p(\overline\h^i):= R^{i} f_* \Omega^{\ge p}_{\Xb|\Bb}(\log D)
\]
of the sheaves $\F^p(\h^i)$ are actually extensions as vector \emph{sub-bundles} of $\overline \h^i$ (this is a particular case of Schmid's Nilpotent Orbit Theorem). The last ingredient is the following classical theorem

\begin{thm}[Manin, \cite{Deligne70}, Prop. 5.4]\label{deligne} Let $(\h, \nabla)$ be an algebraic differential equation with regular singular points on the pointed disc $\Delta^*$. Then $(\h, \nabla)$ admits a \emph{unique} locally free extension $(\overline \h, \overline \nabla)$ to the disc $\Delta$ satisfying the following two properties
\begin{enumerate}
\item $\overline \nabla: \overline \h \to \overline \h \otimes \Omega^1_{\Delta}(\log 0)$ has logarithmic poles;
\item The  eigenvalues $\lambda$ of $\Res_0 (\overline \nabla) \in \End (\overline \h_0)$ satisfy $0\le \re (\lambda)<1$.
\end{enumerate}
\end{thm}
An extension of an algebraic differential equation $\h$ as in the Theorem is called the \emph{canonical extension}. For example, if $(\h, \nabla)=(\mc{O}_{\Delta^*}, d)$, then the trivial extension $(\ov \h=\mc{O}_{\Delta}, \ov \nabla=d)$ has no poles and is the canonical extension.

\begin{thm}[\cite{Katz71}, VII]  \label{Katz monodromy theorem}
Let $\Zb \to \Bb$ be as above, where $D=\Zb \setminus Z$ a reduced divisor with normal crossing. Then for every $b \in \overline B \setminus B$, the extension $(R^if_* \Omega^\bullet_{Z|B}(\log D), \overline \nabla)$, with $\overline \nabla$ as in (\ref{extending nabla}), of $(R^if_* \Omega^\bullet_{Z|B},  \nabla)$  satisfies the assumptions of Theorem \ref{deligne} with the eigenvalues of the residue $\Res_b (\overline \nabla)$ equal to zero. In particular, $R^if_* \Omega^\bullet_{Z|B}(\log D)$ is the canonical extension of $R^if_* \Omega^\bullet_{Z|B}$.
\end{thm}

\begin{cor} \cite[Cor. 11.18]{Peters-Steenbrink} \label{monodromy unipotent}
The monodromy of a one parameter family degenerating to a reduced normal crossing central fiber is unipotent.
\end{cor}
\begin{proof}
This follows immediately from (\ref{monodromy and residue}).
\end{proof}

Under this circumstance we say by abuse of notation that $R^{i} f_* \Omega^{\ge p}_{\Zb|\Bb}(\log D)$ and $R^qf_* \Omega^p_{\Zb|\Bb}(\log D)$  are the \emph{canonical extensions} of $R^if_* \Omega^{\bullet\ge p}_{Z|B}$ and $R^qf_* \Omega^p_{Z|B}$, respectively.
Notice that we can say ``canonical extension'' also for $\F^p(\overline\h^i)$ and for $\Gr^p(\h^i)$ because such extension are uniquely determined by the canonical extension. Indeed, let $j: B \to \Bb$ be the open immersion, let $\h$ be a vector bundle on $B$, and let $E\subset \h$ be a sub-bundle. Suppose we are given a vector bundle extension $\hb$ of $\h$ on the whole of $\Bb$. Any extension of $E$ to $\Bb$ as a sub-bundle of $\hb$ is always contained in the saturation of $\hb \cap j_*E$ in $\hb$, and thus has to be isomorphic to the saturation itself. In particular, since the extension of $E$ as a vector sub-bundle of $\hb$ is unique, so is the extension of the quotient $\h/E$.

We now go back to our situation applying these remarks to the families $\ov{\mc C} \to \Bb$ and $\ov{\mc J}\to \Bb$.

\begin{proof}[Proof of Proposition \ref{degenerations}]
Our aim is to prove that the natural isomorphism $R^i q_* \mc O_{\ov{\mc J}} \res{B} \cong \wedge^i R^1 p_* \mc O_{\ov{\mc C}}\res{B}$ extends over $\Bb$. 
As should be clear by now, we will show this by using the canonical extensions of the VHS associated to the families $p:  \ov{ \mc C} \to \ov{ B}$  and $q: \ov{\mc J} \to \ov{B}$. Indeed by Proposition \ref{MRV on compactified Jac} they both have singular fibers that are reduced and  normal crossing  so that we can apply the theory of degeneration of Hodge bundles. Set
\[
A= \ov B \setminus B, \quad \text{ and } \quad  \ov{\mc C}_A=p^{-1}(A), \,\, \ov{\mc J}_A=q^{-1}(A).
\]
The sheaves
\[
\begin{aligned}
\overline \h_{\ov{\mc J}}^1&= R^1 q_* \Omega^{\bullet}_{\ov{\mc J}|\Bb}(\log ( \ov{\mc J}_A))\\
\overline \h_{\ov{\mc C}}^1&= R^1 p_* \Omega^{\bullet}_{{\ov{\mc J}}|B}(\log ({\ov{\mc C}}_A))
\end{aligned}
\]
both extend $R^1q_* \C\otimes {\mc O}_B\cong {R^1p_*\C\otimes {\mc O}_B}$,  and by Theorem \ref{Katz monodromy theorem} they are both isomorphic to the canonical extension. Hence, there is an isomorphism
\be \label{hX e hC}
\overline \h_{\ov{\mc J}}^1 \cong \overline \h_{\ov{\mc C}}^1,
\ee
which extends the existing isomorphism over $B$.  As a consequence, we  alsoget an isomorphism of the canonical extension of
$\F^1(\h_{\ov{\mc J}}^1)\cong  \F^1(\h_{\ov{\mc C}}^1)$, i.e.,
\[
\F^1(\overline \h_{\ov{\mc J}}^1)\cong \F^1(\overline \h_{\ov{\mc C}}^1),
\]
which in turn implies that there is an isomorphism of the first graded pieces
\[
\Gr^0(\overline \h_{\ov{\mc J}}^1)\cong \Gr^0(\overline \h_{\ov{\mc C}}^1),
\]
i.e., an isomorphism
\[
R^1q_* {\mc O}_{\ov{\mc J}}\cong R^1p_*{\mc O}_{\ov{\mc C}}.
\]

We now want to prove that
\[
R^gq_* {\mc O}_{\ov{\mc J}}\cong \wedge^g R^1p_*{\mc O}_{\ov{\mc J}}.
\]
A connection $\nabla$ on a vector bundle $\h$ naturally induces connections, denoted by $\nabla_j$, on all the exterior powers $\wedge^j \h$ of $\h$ by setting $\nabla_j(h_{i_1} \wedge \dots \wedge h_{i_j})= \sum \pm h_{i_1} \wedge \dots \wedge \nabla(h_{i_k}) \wedge \dots \wedge h_{i_j}.$ By construction, if $(\h, \nabla)$ admits an extension $(\overline \h, \overline \nabla)$ with only logarithmic poles, then so does $(\wedge^j  \h, \nabla_j)$ as we can set $\overline {( \wedge^j \h)}=\wedge^j \overline \h$ with the obvious definition for $\overline \nabla_j$. This also shows that if the eigenvalues of the residue of $\overline \nabla$ are zero, then the same is true for the eigenvalues of $\Res(\overline \nabla_j)$ (indeed, both operator will  be nilpotent). In particular, the sheaf $\wedge^j \overline \h$ is the canonical extension of $\wedge^j  \h$.

Since $\ov{\mc J} \to \Bb$ and $\ov{\mc C} \to \Bb $ have normal crossing boundary, by Theorem \ref{Katz monodromy theorem}, by the discussion above, and by uniqueness of the canonical extension we find  
\be
\wedge^i \, \overline \h^1_{\ov{\mc J}}  \cong \wedge^i \, \overline \h^1_{\ov{\mc C}}.
\ee




 Moreover, since $\mc J \to B$ is a family of abelian varieties, we have a natural isomorphism of VHS $\wedge^j \h^1_{\mc J} \cong \h^g_{\mc J}$. Notice that the Hodge filtration on $\h^g_{\mc J}$ is simpy the exterior power of the filtration on $\h^1_{\mc J}$. Using again the result of Katz, we know that
 \[
\overline \h^g_{\ov{\mc J}}:=R^g q_* \Omega^{\bullet}_{\ov{\mc J}}(\log \ov{\mc J}_A)
\]
is also the canonical extension. It follows that,
\[
\wedge^g \overline \h^1_{\ov{\mc C}}  \cong \overline \h^g_{\ov{\mc J}}.
\]
This induces an isomorphism of the respective Hodge filtrations and thus also of the  respective graded pieces.
Since $\Gr^0(\wedge^g\h^1_{\mc {J}})\cong \wedge^g \Gr^0(\h^1_\mc J)$, we conclude
\[
R^gq_* {\mc O}_{\ov {\mc {J}}}\cong \wedge^g R^1p_*{\mc O}_{\ov{\mc {C}}},
\]
and the proposition is proved.
\end{proof}

The referee pointed out that this proposition can also be proved using  \cite[Cor B]{MRVFMAI}.

\section{The second Betti number} \label{second betti number}

This section is devoted to calculating the second Betti number of the relative compactified Jacobian $N$ (assumptions as in Theorem \ref{thm smooth}). We will assume that $|C|$ satisfies Assumption \ref{assumption} (recall that this assumption holds for primitive linear systems, in particular, it holds also for hyperelliptic linear systems). The strategy is to compare this cohomology group with that of the universal family of curves in the linear system and, in fact, we will prove that the two groups have the same dimension. In the whole section, unless otherwise stated, cohomology should be understood with complex coefficients.

As usual, let $\chi$ be a non-zero integer, consider the Mukai vector
\[
w=(0, [C], \chi),
\]
and assume that $v=(0, D, 2\chi)$ is primitive. Let $A$ be an ample line bundle on $T$  such that $H=f^*A$ is $v$-generic and set
\[
N=N_{w,A}, \quad \text{ and } Y=\Phi(N) \subset M=M_{v,H}.
\]


\begin{thm} \label{betti N}
Let $|C|$ be a linear system of genus $g  \ge 3$ on a general Enriques surface $T$, and let $N$ be as above. Suppose that $|C|$ satisfies Assumption \ref{assumption}, then
\[
h^2(N)=11.
\]
\end{thm}

The case when $|C|$ has genus $2$ is done at the end of this Section. We also recall that thanks to Prop. 4.4. of \cite{yoshioka-enriques}, Assumption \ref{assumption} is satisfied in many cases, for example in low genus and for primitive linear systems.

The proof of the Theorem uses the long exact sequence in cohomology associated to the pair $(N, N_U)$, as done in \cite{Rapagnetta} (for a linear system of curves on a K3 surface), and then relies on the comparison of the local systems associated to the family of curves and its relative compactified Jacobian.

As usual, we denote by $\mc C \subset |C| \times T$ the universal family of curves. Consider the second projection
\[ \label{C to T}
p\colon \mc C \to T,
\]
which is a fibration in $\mathbb P^{g-2}$'s outside of the base locus of $|C|$, over which the fiber is isomorphic to $\mathbb P^{g-1}$. Using this we easily see that 
\be \label{b2 C}
H^1(\mc C)=0, \quad \text{ and  that } \quad H^2(\mc C)= \left\{\begin{aligned}
 &\C^{11} \,\,\,\, \text{if} \,\,\,|C| \text{ not hyperellptic}\\
 & \C^{13} \,\,\,\, \text{if} \,\,\,|C| \text{ is hyperellptic}
\end{aligned}
\right.
\ee


Recall that $U \subset |C|$ denotes the locus parametrizing smooth curves. Notice that if $|C|$ is hyperelliptic then $U$ is strictly contained in the locus $U'$ parametrizing smooth fibers of $\nu: N \to |C|$, since the general fiber of $\nu$ over the two components $\Delta_1$ and $\Delta_2$ is smooth.

\begin{lemma} \label{H1} Let $|C|$ be a genus $g \ge 2$ linear system on a general Enriques surface $T$, and let $k$ be the number of irreducible components of $\Delta$. Then $H^1(\mc C_U)\cong H^1(N_U)\cong H^1(U)=\C^{k-1}$.
\end{lemma}

\begin{proof}
The equality $H^1(U)=\C^{k-1}$ is well--known, see for example \cite[Prop. 1.3]{Dimca}. Since $q: \mc C_U \to U$ is a smooth morphism,
\[
H^1(\mc C_U)\cong H^0(U, R^1q_* \C)\oplus H^1(U),
\]
so we only need to show that $H^0(U, R^1q_* \C)=0$. This follows from the invariant cycle theorem \cite{DeligneHodgeII}, which in our setting asserts that
\[
H^0(U, R^1q_* \C)=H^1(C_t)^{\inv}=\im[H^1(\mc C) \to H^1(\mc C_U)],
\]
where $H^1(C_t)^{\inv}$ denotes the monodromy invariant part of the first cohomology of a smooth curve $C_t$. Since $H^1(\mc C)=0$, we are done. To finish the proof we only need to invoke Lemma \ref{higher direct images different degree} which guarantees that 
\[
R^1q_* \C_{\mc C}\res{U}=R^1\nu_* \C_N \res{U}.
\]
\end{proof}

\begin{rem}
By the results of Subsection \ref{linear systems enriques} (cf. Corollary \ref{discriminant locus irr} and Proposition \ref{discriminant locus}), $k=1$ unless either $|C|$ is hyperelliptic, in which case $k=4$, or $|C|$ is of genus $3$ and defines a degree $4$ morphism to $\mathbb P^2$ as in case $(1)$ of Proposition \ref{Verra}, in which case $k=37$.
\end{rem}

We will now consider the long exact sequences in cohomology for the pairs $(N, N_U)$ and $(\mc C, \mc C_U)$. Set
\[
\begin{aligned}
j=\#\text{ of irreducible components of } N_{\Delta},\\
\ell=\#\text{ of irreducible components of } C_{\Delta}.
\end{aligned}
\]
So, for example, if $|C|$ is non--hyperellptic $k=j=\ell$. Indeed, by Proposition \ref{Verra} the general point of every component of the discriminant parametrizes irreducible curves and hence the preimage in $N$ of every component of the discriminant is irreducible. In Corollary \ref{discriminant locus irr} we remarked that in the non--hyperelliptic case $k=1$, unless $|C|$ is of genus $3$ in which case $k=37$. If $|C|$ is hyperelliptic, of genus $g \ge 3$, then by Propositions \ref{discriminant locus} and Corollary \ref{number of components}, $j=5$ and $\ell=7$.

\begin{lemma} \label{H2} There are exact sequences
\be \label{relative cohomology N}
\begin{aligned}
0 \to \C^{j-k+1} \to H^2(N) \to H^2(N_U),\\
0 \to \C^{\ell-k+1} \to H^2(\mc C) \to H^2(\mc C_U).
\end{aligned}
\ee
\end{lemma}
\begin{proof}
By the long exact sequence in  cohomology of  pairs and by the previous Lemma, we only  need to show that $H^2(N, N_U)=\C^{j}$ and that $H^2(\mc C, \mc C_U)=\C^\ell$. By Poincar\'e-Lefschetz duality (\cite{Spanier} Chapter 6, Section \S 2 Thm 17 or also \cite{Peters-Steenbrink} Thm B.28),
\[
H_i(N, N_U) \cong H^{2n-i}(N_\Delta).
\]
Letting $S(N_\Delta)$ be the singular locus of $N_\Delta$ and setting $i=2$ we get, since the real codimension of $S(N_\Delta)$ in $N_\Delta$ is greater or equal to $2$,
\be \label{relative cohomology}
H^{2n-2}(N_\Delta)\cong H^{2n-2}(N_\Delta, S(N_\Delta)) \cong H_0(N_\Delta\setminus S(N_\Delta))\cong \C^j,
\ee
where the second to last isomorphism is again given by Poincar\'e-Lefschetz duality\footnote{cit. Spanier, Thm 19 plus the fact the the pair $(N_\Delta, S(N_\Delta))$ is taut in $N$.} and the last isomorphism holds because by assumption $N_\Delta$ has $j$ irreducible components. The same argument applies for showing that  $H^2(\mc C, \mc C_U)=\C^\ell$.
\end{proof}

\begin{cor} \label{kernel of N to NU}
The dimension of the kernel of the natural morphism $H^2(N) \to H^2(N_U)$ is equal to $1$ if $|C|$ is not hyperelliptic and equal to $2$ if $|C|$ is hyperelliptic.
\end{cor}

\begin{rem}
This Corollary and the Lemma before are the places where we are using Assumption \ref{assumption}. Clearly, if this assumption does not hold the number of irreducible components of $N_\Delta$ cannot be computed in the same way.
\end{rem}


Recall that the degree $i$ cohomology groups of a smooth quasi-projective variety $Z$ are endowed with a canonical mixed Hodge structure (MHS for short) \cite{DeligneHodgeII} of weight $\ge i$. Moreover, if $\ov Z$ is a smooth projective compactification of $Z$,  by  Corollary 3.2.17 of \cite{DeligneHodgeII} we have
\[
W_i H^i(Z)=\im [ H^i(\ov Z) \to H^i(Z)]
\]
where $W_i H^i(Z)$ denotes the weight $i$ part of the MHS on $H^i(\ov Z)$.
Applying this to $N_U$ and $\mc C_U$ we find that there are short exact sequence
\[
\begin{aligned}
0 \to \C^{j-k+1} \to H^2(N) \to W_2H^2(N_U) \to 0,\\
0 \to \C^{\ell-k+1} \to H^2(\mc C) \to W_2H^2(\mc C_U) \to 0.
\end{aligned}
\]

The theorem will follow once we prove the following
\begin{prop} \label{H2 MHS}
There is an isomorphism of MHS
\[
H^2(N_U) \cong H^2(\mc C_U)
\]
\end{prop}

\begin{proof}[Proof of Theorem \ref{betti N}]
By (\ref{b2 C}) we see that $W_2H^2(\mc C_U)=W_2(H^2(N_U))$ is equal to $10$ in the non--hyperelliptic case and equal to $9$ in the hyper--elliptic case while by Corollary \ref{kernel of N to NU} see that $j-k+1$ is equal to $1$ in the non--hyperelliptic case and equal to $2$ in the hyperelliptic case. In either case, we see that $\dim H^2(N)=11$.
\end{proof}

The proof of  Proposition \ref{H2 MHS} will take the rest of this section. We start with the following general statement

\begin{prop} \label{J and J' MHS}
Let $p:\mc C \to B$ be a family of smooth connected curves of genus $g$ over a smooth quasi--projective variety $B$. For every pair of integers $d$ and $d' $ let $q: \mc J=\mc J_d \to B$ and $q': \mc J'=\mc J_{d'} \to B$ be the relative Jacobians of degree $d$ and $d'$, respectively. There exists cycles $Z \in CH^g(\mc J \times_B \mc J')_\Q$ and $P \in CH^2(\mc C \times_B \mc J)_\Q$ inducing
\begin{enumerate}
\item natural isomorphisms $[Z]_*: R^k q_* \Q_{\mc J} \cong R^k q'_* \Q_{\mc J'}$ of local systems for every $k$;
\item an isomorphism of MHS $[Z]_*: H^k(\mc J) \cong H^k(\mc J')$ for every $k$.
\item natural morphisms $[P]_*: R^k q_* \Q_{\mc J} \to R^k p_* \Q_{\mc C}$ of local systems for every $k$ (isomorphism for $k=1$);
\item a morphism of MHS $[P]_*: H^k(\mc J) \to H^k(\mc C)$ for every $k$.
\end{enumerate}
Moreover, the morphisms of MHS  in $(2)$ and $(4)$ are compatible with smooth base change and with the Leray filtrations of the two sides.
\end{prop}
\begin{proof} Let $\Sigma \subset \mc C$ be a multisection of the family and let  $\phi: \Sigma \to B$ be the induced morphism. Let $r$ be its degree. If we base change $\mc C$ to $\Sigma$, there is a tautological section and hence all  relative Jacobians are isomorphic. In particular, we can find the graph of an isomorphism $\Gamma \subset (\mc J\times_B {\Sigma}) \times_\Sigma (\mc J'\times_B \Sigma)$. Let $\xi:(\mc J\times_B {\Sigma}) \times_\Sigma (\mc J'\times_B \Sigma) \to \mc J\times_B \mc J'$ be the natural projection and set
\[
Z:= \ff{1}{r} \xi_*\Gamma \in CH^g(\mc J \times_B \mc J')_\Q.
\]
This cycle can be viewed also as an element of $H^0(B, R^{2g} (q,q')_* \Q)=\oplus H^0(R^{2g-k}q_* \Q_{\mc J} \otimes R^k q'_* \Q_{\mc J'})$, where $(q, q'): \mc J\times_B \mc J' \to B$ is the natural morphism. This defines for every $k$ an element in $H^0(B, R^{2g-k}q_* \Q_{\mc J} \otimes R^k q'_* \Q_{\mc J'})= Hom_B(R^k q_* \Q_{\mc J} , R^k q'_* \Q_{\mc J'})$, hence a morphism 
\[
[Z]_*: R^k q_* \Q_{\mc J} \to R^k q'_* \Q_{\mc J'}
\]
which by \cite[Lemma 5.4]{Arapura05}, $[Z]_*$ is the composition
\be \label{lemma arapura}
R^k q_* \Q_{\mc J} \stackrel{p_1^*}{\to} R^{k} (q,q')_* \Q \stackrel{\cup Z}{\to} R^{k+2g} (q,q')_* \Q \stackrel{{p_2}_*}{\to} R^k q'_* \Q_{\mc J'}),
\ee
where $p_1$ and $p_2$ are the first and second projection from $ \mc J\times_B \mc J' $. To show this is an isomorphism let us look at the stalks of $[Z]_*$.
Let $b \in B$ be a point and let $\sigma \in \Sigma$ be such that $\phi(\sigma)=b$.  
By (\ref{lemma arapura}), the stalk at $b$ of the morphism $[Z]_*$ is precisely the correspondence $[Z_b]_*$ induced by the cycle $Z_b=(\ff{1}{r}\xi_*\Gamma)_b$. To understand what this is, recall that by construction, the isomorphism $[\Gamma_\sigma]_*: H^k(\mc J_b) \to H^k(\mc J'_b)$ is the isomorphism in cohomology induced by
\[
\begin{aligned}
\mc J_b &\to \mc J'_b,\\
L &\mapsto L \otimes \mc O_{\mc C_b}((d'-d)\sigma ).
\end{aligned}
\] 
As in Lemma \ref{higher direct images different degree} we can see that $[\Gamma_\sigma]_*$ is independent of  the point $\sigma$ for $\sigma \in \phi^{-1}(b)$. 
Since $[\Gamma_\sigma]_*$ is independent of $\sigma$ this implies that $[Z_b]_*=r[\ff{1}{r}\Gamma_\sigma]_*: H^k(\mc J_b) \to H^k(\mc J'_b)$. In particular, $[Z]_*$ is an isomorphism.

Under the cycle map $CH^g(\mc J \times_B \mc J')_\Q \to H^{BM}_{2\dim B+2g}(\mc J \times_B \mc J') $ we can also view   $Z$ as class in Borel--Moore homology. This defines a map
\[
\begin{aligned}
{[Z]_*}: H^k(\mc J) &\to H^k(\mc J')\\
\alpha & \mapsto {p_2}_*( [Z] \cup p_1^*(\alpha))
\end{aligned}
\]
By compatibility of MHS with cup--product, this is a morphism of MHS (e.g. see \cite[\S 6.3]{Peters-Steenbrink}).  Lemmas 5.2 and 5.3 of \cite{Arapura05} show that the Leray filtration is compatible both with cup product and with pushforward under smooth projective morphisms as is $p_2$. This shows (cf. \cite{Arapura05} pg. 586) that $[Z]_*$ is compatible with the Leray filtrations on both $H^k(\mc J)$ and $H^k(\mc J')$ and hence induces morphisms
\be \label{}
H^i(B, R^j q_* \Q_{\mc J} ) \to H^i(B, R^j q'_* \Q_{\mc J'} ).
\ee
which by $(1)$ are isomorphisms.
This shows that $(2)$ is an isomorphism as well.
 
For the last two statements, we may consider the cycle $P \in CH^2(\mc C \times_B \mc J)_\Q$ obtained by considering the Poincar\'e line bundle on $(\mc C\times_B {\Sigma}) \times _\Sigma (\mc J_r \times_B \Sigma)$,  pushing it forward to $\mc C \times_B \mc J_r$ and then dividing by $r$. With this definition, the proof of $(3)$ and $(4)$, and of compatibility with the Leray filtrations, follows from the general theory as in $(1)$ and $(2)$.
\end{proof}

\begin{proof}[Proof of Proposition \ref{H2 MHS}]
By the last statement of Proposition \ref{J and J' MHS}, there is a morphism of MHS
\[
\phi: H^2(N_U) \to H^2(\mc C_U)
\]
which is compatible with the Leray filtrations of $N_U \to U$ and $\mc C_U \to U$. To show that $\phi$ is an isomorphism, it is sufficient to show that the natural morphism of local systems of Proposition \ref{J and J' MHS} induce isomorphisms
\be \label{iso local systems}
H^i(U, R^j\nu_* \C_N) \cong H^i(U, R^jq_* \C_{\mc C}).
\ee
Since for $j=1$ there is an isomorphism
\[
R^1 \nu_* \C_{N_U}\cong  R^1 p_* \C_{\mc C_U},
\]
we only need to prove (\ref{iso local systems}) for $(i,j)=(0,2)$. By Proposition \ref{J and J' MHS} there is a (non--zero) morphism of local systems  $R^2\nu_* \C_{N_U}) \to R^2 p_* \C_{\mc C_U}$ so we only have to show that it induces an isomorphism at the level of global sections. By the invariant cycle theorem we know that
\[
H^0(U, R^2\nu_* \C_{N_U})=H^2(N_t, \C)^{\inv}, \quad \text{ and } \quad H^0(U, R^2 p_* \C_{\mc C_U})=H^2(C_t, \C)^{\inv}.
\]
where, since $C_U \to U$ is a family of smooth connected curves, we known that $H^0(U, R^2 p_* \C_{\mc C_U})$ is one dimensional. So we only have to prove that $H^0(U, R^2\nu_* \C_{N_U})$, too, is one dimensional. Let 
\[
\rho: \pi_1(U) \to \Aut(H^1(C_t, \C))
\]
be the monodromy representation. By Proposition \ref{irreducible monodromy non hyp} below, this representation is irreducible.
Since $\rho$ preserves the symplectic pairing $(\cdot,\cdot)$, the isomorphism  $w: H^1(C_t, \C) \cong H_1(C_t, \C)$ induced by $(\cdot,\cdot)$ is $\rho$-equivariant. By composing with $w$, it follows that any $\rho$-invariant element of $ \wedge^2H^1(C_t, \C)$ can be thought of as a $\rho$-invariant morphism 
\[
\varphi: H^1(C_t, \C) \to H^1(C_t, \C).
\]
By Schur's Lemma
\[
\varphi=\lambda \id,
\]
for some $\lambda$ in $\C$, and 
\[
( \wedge^2H^1(C_t, \C))^{\inv}\cong \C.
\]
is one--dimensional, generated by the class of the the intersection pairing (viewed as the theta divisor of $N_t$ via the natural isomorphism $\wedge^2H^1(C_t, \C) \cong H^2(N_t, \C)$).
\end{proof}

\begin{prop} \label{irreducible monodromy} \label{irreducible monodromy non hyp}
Let $|C|$ be a linear system of genus $g \ge 2$ on a general Enriques surface $T$, and let $t \in U \subset |C|$ be a point. The monodromy representation
\be
\rho: \pi_1(U) \to \Aut(H^1(C_t, \C)),
\ee
is irreducible.
\end{prop}
\begin{proof} 
Let $N$ be the degree zero relative compactified Jacobian of $|C|$.  By Corollary \ref{generated by vanishing cycles} we know that 
the kernel of the natural morphism
\[
\pi_1(N_t)=H^1(C_t, \Z) \to \pi_1(N) \to 1,
\]
is generated by vanishing cycles. Hence so is $H^1(C_t, \C)$ ( we freely identify $H^1(C_t, \C)$ with $H_1(C_t, \C)$ with the monodromy invariant isomorphism defined by Poincar\'e duality).
If the discriminant locus $\Delta \subset |C|$ is irreducible, then we can conclude using Theorem 3.4 in \cite{Voisin2} which shows that the restriction of the monodromy representation to the subspace generated by the vanishing cycles is irreducible provided that $\Delta$ is irreducible. If not, we argue as follows.
Let
\[
\{c_i\},
\]
be the set of vanishing cycles associated to a set of generators $\{ h_i \}$ of $H=\pi_1(U)$ as in Subsection \ref{section vanishing cycles}.

To prove that there are no invariant subspaces, we argue by contradiction and suppose that there is a non-trivial invariant subspace
\[
F \subset H^1(C_t, \C).
\]
First of all, we check that the intersection pairing $(\cdot, \cdot)$ on $H^1(C_t, \C)$ restricts to a symplectic pairing on $F$. Indeed, since the $\{c_i\}$ generate $H^1(C_t, \C)$ and the intersection pairing is non-degenerate, for every non-zero $\beta \in F$ there exist an $i$ such that $(\beta, c_i) \neq 0$. Since $F$ is a $\rho$-invariant subspace, it follows that the the image of $\beta$ under the Picard--Lefschetz monodromy transformation along $h_i$, which is $\rho_{h_i}(\beta)=PL_{h_i}(\beta)=-\beta+(\beta,\alpha_{i})\alpha_{i}$, lies in $F$, and thus
\[
\alpha_{i} \in F.
\]
Hence $F$ and its orthogonal complement $F^\perp$, which also is monodromy invariant, are two symplectic vector spaces. Set
\[
2n=\dim F, \quad 2m=\dim F^\perp.
\]
The above argument also shows that every vanishing cycle $ c_i$ lies either in $F$ or in $F^\perp$. In particular, we can decompose
\[
R:=\ker[\pi_1(N_t) \to \pi_1(N)]=R_F\oplus R_{F^\perp},
\]
where $R_F$ and $R_{F^\perp}$ are the (non--degenerate) sublattices generated by the vanishing cycles that lie in $F$ and $F^\perp$ respectively.
Since $R\subset H^1(C_t, \Z)$ has index two and $( \cdot, \cdot)$ is unimodular, it follows that the determinant of the intersection matrix for $R$ is $4$ (cf. \cite{BHPV} $\S$ 1.2). Since $R_F$ and $R_{F^\perp}$ are symplectic lattices, the determinants of their intersection matrices are squares and hence, up to switching $F$ and $F^\perp$, we may assume that one determinant is $1$ and the other is $4$. In particular, up to switching $F$ and $F^\perp$,  we may assume that $R_F \subset H^1(C_t, \Z)$ is primitive. Set $F_\Z:=R_F$ and $E_\Z:=(F_\Z)^\perp$,  where the orthogonal complement is taken in $H^1(C_t, \Z)$ so that we have decomposed
\be
H^1(C_t, \Z)=F_\Z \oplus E_\Z.
\ee
as a direct sum of two primitive lattices.
Consider now the following commutative diagram, 
\[
\xymatrix{
H^1(C_t, \Z) \ar@{^{(}->}[r]^j \ar@{^{(}->}[d] & H^1(C_t, {\mc O}_{C_t}) \ar[d]_D^\sim\\
H^1_{dR}(C_t, \C) \ar@{->>}[r]_p & H^{0,1}_{\overline \partial }(C_t)
}
\]
where the left hand side vertical arrow is the composition of the base change inclusion $H^1(C_t, \Z)\subset H^1(C_t, \C)$, and the De Rham isomorphism. The top horizontal arrow is given by the exponential sequence, the right hand side vertical arrow is the Dolbeaut isomorphism, and the bottom arrow is the projection onto the Dolbeaut group.
The two spaces
\[
F':=D^{-1}p(F) \quad \text{ and} \quad  E':=D^{-1}p(E),
\]
contain the lattices $j(F_\Z)$ and $j(E_\Z)$, of rank respectively equal to $2n$ and $2m$ and are thus of dimension equal to $n$ and $m$ respectively. It follows that
\[
F_t:= F'/j(F_\Z), \quad \text{ and }E_t:= E'/j(E_\Z),
\]
are two smooth abelian varieties of dimensions equal to $n$ and $m$ respectively. They are both principally polarized since $F_\Z$ and $E_\Z$ are unimodular. In particular
\be \label{decomposition ppav}
\Jac^0(C_t) \cong F_t \times E_t.
\ee
However, since the intersection product on $H^1(C_t, \Z)$ can be viewed, via the isomorphism $\wedge^2 H^1(C_t, \Z) \cong H^2(\Jac^0(C_t) , \Z)$, as the theta divisor of $\Jac^0(C_t)$, it follows that (\ref{decomposition ppav}) is actually a decomposition as principally polarized abelian varieties. We have thus reached a contradiction since the Jacobian variety of a smooth curve is irreducible as a principally polarized abelian variety.

\end{proof}

\subsection{The genus two case} \label{genus two case} 
When $|C|$ is a genus two linear system on a general Enriques surface $T$, we get a  Calabi-Yau three-fold whose Hodge numbers are described by the following theorem.

\begin{thm} \label{genus two thm}
Let $|C|$ be a genus two linear system on a general Enriques surface $T$. Then $\nu: N \to \P^1=|C|$  has exactly $16$ singular fibers, each of which is a rank one degeneration of an abelian surface, $\pi_1(N)=\Z/(2)$, and the Hodge diamond of $N$ is the following:
\[
\begin{array}{cccccc}
&1&&\\
&0\,\,\,\,\,\,0&&\\
\,\,\,\,0&10&0\,\,\,\,&\\
1\,\,\,\,&10\,\,\,\,\,\,10&\,\,\,\,1&\\\
\end{array}
\]
The singular fibers of the natural abelian surface fibration on $\wt N$ are of the same kind, and $\wt N$ has the same Hodge numbers.
\end{thm}
\begin{proof}
Since $N$ is smooth and $|C|$ is one-dimensional, the support morphism is flat and thus Theorem \ref{fundamental group} holds unconditionally.
As for the singular fibers, this follows from Proposition \ref{discriminant locus} and the fact that the Jacobian of a union of two smooth curves meeting transversally in one point is smooth and projective and the fact that the compactified Jacobian of an irreducible nodal curve is a rank one degeneration of an abelian variety. As for the second Betti number, the only difference is that now
\[
H^1(\mc C_U)\cong H^1(N_U)\cong H^1(U)=\C^{17},   \quad \text{ and } \quad  H^2(\mc C)\cong \C^{12}.
\]
Moreover,
\[
\begin{aligned}
h^2(\Nt, \Nt_U)=h^2(Y, Y_U)=h^2(N, N_U)&=\#\text{ of irreducible components of } N_{\Delta}=18,\\
h^2(\mc C, \mc C_U)&=\#\text{ of irreducible components of } \mc C_{\Delta}=20.
\end{aligned}
\]
so that (notation as above)  $\ell-k+1=3$ and $j-k+1=1$. Hence
\[
W_2(H^2( N_U))=W_2(H^2(\mc C_U))=\C^{9}
\]
and
\[
H^2(N)\cong \C^{10}.
\]

To compute the remaining Hodge number $H^{2,1}(N)$, it is sufficient to notice that $\chi_{top}(N)=0$ and hence $2H^2+2=H^3=3+H^{2,1}$. This follows from the fact that the all the fibers of $N  \to |C|$ have trivial topological Euler characteristic.

\end{proof}

The fact that the second Betti number group of this three dimensional moduli space is $10$ reminds of what happens in the case of K3 surfaces. Indeed, the second Betti number of the higher dimensional examples is equal to $23=b_2(K3)+1$, whereas that of the two dimensional moduli spaces is equal to $22$.

\bibliography{bibliografiagiulia}
\bibliographystyle{alpha}
\end{document}